\documentclass{article}
\usepackage{amsmath, amssymb, amsthm, epsf, graphicx, amscd, amsfonts, amscd, multirow}

\usepackage{setspace}
\newtheorem{theorem}{Theorem}[section]
\newtheorem{lemma}[theorem]{Lemma}
\newtheorem{proposition}[theorem]{Proposition}

\def\bb #1{ {\mathbb #1} }
\def\c #1{ {\mathcal #1} }
\def\f #1{ {\mathfrak #1} }
\def\b #1{ {\bf #1} }

\begin{document}
\title{$L$-functions associated with families of toric exponential sums}

\author{C. Douglas Haessig\footnote{Partially supported by NSF grant DMS-0901542} \and Steven Sperber}

\date{\today}
\maketitle

\begin{abstract}
We consider arbitrary algebraic families of lower order deformations of nondegenerate toric
exponential sums over a finite field. We construct a relative polytope with the aid of which we define
a ring of coefficients consisting of $p$-adic analytic functions with polyhedral growth prescribed by the
relative polytope. Using this we compute relative cohomology for such families and calculate sharp
estimates for the relative Frobenius map. In applications one is interested in $L$-functions associated
with linear algebra operations (symmetric powers, tensor powers, exterior powers and combinations
thereof) applied to relative Frobenius. Using methods pioneered by Ax, Katz and Bombieri we prove
estimates for the degree and total degree of the associated $L$-function and $p$-divisibility of the reciprocal
zeros and poles. Similar estimates are then established for affine families and pure Archimedean weight
families (in the simplicial case).
\end{abstract}

\section{Introduction}

Symmetric power $L$-functions and their variants have long been objects of study and have been valuable in many number theoretic applications. In the function field case, these $L$-functions are associated with families of varieties or families of exponential sums defined over a base space $S$ which is itself a variety over $\bb F_q$, the finite field of $q = p^a$ elements having characteristic $p$. For each closed point $s$ of $S$, there is a zeta function or $L$-function for the fibre over this point. This is a rational function for each such $s$, and the collection $\c A(s)$ of reciprocal zeros and poles of this function is a finite set of algebraic integers. Interesting new $L$-functions may be created by taking Euler products as follows. Let $\c A_0(s) \subset \c A(s)$ be an interesting well-chosen subset for each point $s \in |S / \bb F_q|$, the set of closed points of $S$. For example, some choices include taking the subset of $\c A(s)$ consisting of all $p$-adic units, or all elements of $\c A(s)$ having a fixed archimedean weight. Once chosen, we can form symmetric, tensor, or exterior powers  (or combinations thereof) of the elements of $\c A_0(s)$ and denote the resulting set by $\c L \c A_0(s)$. Then the Euler product we are interested in has the general form
\begin{equation}\label{E: GenL}
L(\c L \c A_0, S / \bb F_q, T) := \prod_{s \in |S / \bb F_q|} \ \ \prod_{\tau(s) \in \c L \c A_0(s)} (1 - \tau(s) T^{deg(s)})^{-1},
\end{equation}
where $deg(s) := [ \bb F_q(s): \bb F_q]$ is the degree of the point.

The $p$-adic study of symmetric power $L$-functions begins with Dwork \cite{Dwork-Heckepolynomials-1971}, whose work was itself inspired by Ihara \cite{Ihara-Hecke_polynomials} and Morita \cite{Morita-Hecke_polynomials}, who linked the Ramanujan-Petersson conjectures and the Weil conjectures. Implicit in their work were symmetric power $L$-functions for a suitable family of elliptic curves. In \cite{Dwork-Heckepolynomials-1971}, Dwork explicitly considered symmetric power $L$-functions associated with the Legendre family of elliptic curves, applying $p$-adic cohomology to obtain important information about this $L$-function. This study was continued by Adolphson \cite{Adolphson-$p$-adictheoryof-1976} to obtain congruence information. Adolphson also considered symmetric power $L$-functions for the family of elliptic curves with level three structure \cite{Adolphson-level3}. Based on Dwork's $p$-adic cohomology theory of the Bessel function \cite{Dwork-Besselfunctionsas-1974}, Robba \cite{Robba-SymmetricPowersof-1986} $p$-adically studied the symmetric power $L$-functions of the family of Kloosterman sums. A similar study was made in Haessig \cite{Haessig-$L$-functionsofsymmetric-} for cubic exponential sums.

Symmetric, tensor, and exterior power $L$-functions have also been studied using $\ell$-adic techniques going back at least to the work of Deligne. Katz \cite{Katz-Moments_monodromy_and_perversity} has studied symmetric power $L$-functions of families of elliptic curves and their monodromy behavior. In recent work, Fu and Wan have obtained very detailed information on symmetric power $L$-functions of hyper-Kloosterman sums \cite{FuWan-$L$-functionssymmetricproducts-2005} \cite{FuWan-TrivialfactorsL-functions-}. In \cite{HaesRoja-$L$-functionsofsymmetric-}, Haessig and Rojas-Le\'on studied the $k$-th symmetric power $L$-functions for a one-parameter family of exponential sums in one variable.

In another direction, Wan \cite{Wan-Dwork'sconjectureunit-1999} \cite{Wan-Higherrankcase-2000} \cite{Wan-Rankonecase-2000} proved a conjecture of Dwork's that unit root $L$-functions which come from geometry are $p$-adically meromorphic by relating the unit root $L$-function to symmetric, tensor, and exterior power $L$-functions via Adams operations and then employing a suitable $p$-adic limiting argument.

In the present work, we consider a general family of nondegenerate toric exponential sums. Let
\[
\bar G(x, t) := \bar f(x) + \bar P(x, t) \in \bb F_q[x_1^\pm, \ldots, x_n^\pm, t_1^\pm, \ldots, t_s^\pm]
\]
where $\bar f(x)$ is nondegenerate with respect to $\Delta_\infty(\bar f)$, its Newton polyhedron at $\infty$. We assume as well that the monomials in the $x$-variables in $\bar P(x, t)$ have strictly smaller (polyhedral) weight than the leading monomials in $\bar f$, and that the $dim \> \Delta_\infty(\bar f) = n$. As a consequence, for each choice $\lambda  \in \overline{\bb F}_q^{*s}$, an $\overline{\bb F}_q$-rational point of $\bb G_m^s$ where $\overline{\bb F}_q$ is an algebraic closure of $\bb F_q$, the Newton polyhedron at $\infty$ of $\bar G(x, \lambda)$ is nondegenerate with respect to $\Delta_\infty(\bar G(x, \lambda)) = \Delta_\infty(\bar f)$. Let $\bb F_q(\lambda)$ be the field of definition of $\lambda$ and $deg(\lambda) := [\bb F_q(\lambda) : \bb F_q]$ its degree. Let $Cone(\Delta)$ be the cone over $\Delta := \Delta_\infty(\bar f)$, and $M(\Delta) := \bb Z^n \cap Cone(\Delta)$. Fix an additive character $\Theta$ of $\bb F_q$, and set $\Theta_\lambda := \Theta \circ Tr_{\bb F_q(\lambda) / \bb F_q}$. Let $\bar G_\lambda(x) := \bar G(x, \lambda)$. Define the exponential sums
\[
S_r(\bar G_\lambda, \Theta, \bb G_m^n / \bb F_q(\lambda)) := \sum_{x \in (\bb F_{q^{r deg(\lambda)}}^* )^n } \Theta_\lambda \circ Tr_{\bb F_{q^{r deg(\lambda)} }/ \bb F_q(\lambda)} \bar G_\lambda(x)
\]
and the associated $L$-function
\[
L( \bar G_\lambda, \Theta, \bb G_m^n / \bb F_q(\lambda), T) := \exp \left( \sum_{r=1}^\infty S_r(\bar G_\lambda, \Theta, \bb G_m^n / \bb F_q(\lambda)) \frac{T^r}{r} \right).
\]
When there is no confusion, we will denote the above exponential sums by $S_r(\lambda)$ and the associated $L$-functions by $L( \bar G_\lambda, T)$. $L$-functions of this type have been studied in \cite{AdolpSperb-ExponentialSumsand-1989} where it was shown that $L(\bar G_\lambda, T)^{(-1)^{n+1}}$ is a polynomial of degree $N := n! \> vol \Delta_\infty(\bar f)$ whose coefficients lie in the cyclotomic field $\bb Q(\zeta_p)$ of $p$-th roots of unity, and the reciprocal zeros are algebraic integers. Using the $p$-adic absolute value normalized at the fiber over $\lambda$ by requiring $ord_\lambda(q^{deg(\lambda)}) = 1$, then the lower bound for the Newton polygon of $L(\bar G_\lambda, T)^{(-1)^{n+1}}$, calculated using $ord_\lambda$, is independent of $\lambda$ and given in \cite{AdolpSperb-ExponentialSumsand-1989}, as well as in (\ref{E: NPfibre}) below. Denef and Loeser \cite{Denef-Loeser-Weights_of_exponential_sums} have given a precise description of the distribution of archimedean weights for the reciprocal zeros of $L(\bar G_\lambda, T)^{(-1)^{n+1}}$. The results of Denef-Loeser are independent of $\lambda$ as well.

We write then for each $\lambda \in \overline{\bb F}_q^{*s}$,
\[
L(\bar G_\lambda, T)^{(-1)^{n+1}} = (1 - \pi_1(\lambda) T) \cdots (1 - \pi_N(\lambda)T).
\]
Let $\c A(\lambda) := \{ \pi_i(\lambda) \}_{i=1}^N$. To fix ideas, we present some examples of $L$-functions of the form (\ref{E: GenL}). The $k$-th tensor power $L$-function of the toric family above is the Euler product
\[
L(\c A^{\otimes k}, \bb G_m^s / \bb F_q, T) := \prod_{\lambda \in |\bb G_m^s / \bb F_q|} \prod (1 - \pi_{i_1}(\lambda) \pi_{i_2}(\lambda) \cdots \pi_{i_k}(\lambda) T^{deg(\lambda)})^{-1},
\]
where the inner product on the right runs over all $k$-tuples $(i_1, \ldots, i_k) \in S^k$ with $S := \{1, 2, \ldots, N\}$. Similarly the $k$-th symmetric power $L$-function is given by
\[
L( Sym ^k \c A, \bb G_m^s / \bb F_q, T) := \prod_{\lambda \in |\bb G_m^s / \bb F_q|} \prod (1 - \pi_1(\lambda)^{i_1} \cdots \pi_N(\lambda)^{i_N} T^{deg(\lambda)})^{-1}
\]
where the inner product runs over $N$-tuples of non-negative integers $(i_1, \ldots, i_N)$ satisfying $i_1 + \cdots  + i_N = k$. Another variant of interest focuses on the  subset $\c A_0(\lambda) \subset \c A(\lambda)$ consisting of the unique $p$-adic unit root, say $\pi_0(\lambda)$, in $\c A(\lambda)$. Then the $k$-th moment unit root $L$-function is defined by
\begin{equation}\label{E: DworkUnit}
L_{\text{unit}}( k, \bar G, \bb G_m^s / \bb F_q, T) := \prod_{\lambda \in |\bb G_m^s / \bb F_q|} (1 - \pi_0(\lambda)^k T^{deg(\lambda)})^{-1}.
\end{equation}
We may also denote by $\c W_i(\lambda)$ the subset of $\c A(\lambda)$ consisting of reciprocal zeros $\pi(\lambda)$ having archimedean weight equal to $i$, that is, $|\pi(\lambda)|_{\bb C} = q^{deg(\lambda) i / 2}$. Then, we define
\[
L( \c L \c W_i, \bb G_m^s / \bb F_q, T) := \prod_{\lambda \in |\bb G_m^s / \bb F_q|} \prod_{\tau(\lambda) \in \c L \c W_i(\lambda)} (1 - \tau(\lambda) T^{deg(\lambda)})^{-1}.
\]

We now state our main result for the full toric family $\c A := \coprod_{\lambda \in | \bb G_m^s / \bb F_q|} \c A(\lambda)$. Similar statements about other families may be found in the following sections. Section \ref{SS: Mixed} looks at a family of affine exponential sums (in fact it deals more generally with mixed toric and affine sums), Section \ref{SS: Pure} looks at a family of pure weight (in the archimedean sense), and Section \ref{SS: Unitfam} looks at a $p$-adic unit root family. Let $\c L N$ denote the cardinality of the set $\c L \c A(\lambda)$; this number is independent of the choice of $\lambda$. Let $\Gamma \subset \bb R^s$ be the relative polytope of $\bar G$, as defined as follows (see also (\ref{E: RelPolyDef})). Let $w$ be the polyhedral weight function defined by $\Delta_\infty(\bar f)$ in $\bb R^n$. Define
\begin{equation}\label{E: GammaRelPoly}
\Gamma := \text{ Convex hull in $\bb R^s$ of the points } \{0 \} \cup \left\{ \left(\frac{1}{1 - w(\mu)} \right) \gamma \in \bb Q^s \mid (\gamma, \mu) \in Supp(\bar P) \right\}.
\end{equation}
Let $\tilde s$ denote the dimension of the smallest linear subspace of $\bb R^s$ which contains $\Gamma$, and denote by $vol(\Gamma)$ the volume of $\Gamma$ in this linear subspace with respect to Haar measure normalized so that a fundamental domain of the integer lattice in the subspace has unit volume. Lastly, we define the order $| \c L| := r$ of a linear algebra operation $\c L$ as the least positive integer $r$ such that $\c L$ is a quotient of an $r$-fold tensor product.

\begin{theorem}\label{T: IntroMain}
For each linear algebra operation $\c L$, the $L$-function $L(\c L\c A, \bb G_m^s / \bb F_q, T)$ is a rational function:
\[
L(\c L \c A, \bb G_m^s / \bb F_q, T)^{(-1)^{s+1}} = \frac{ \prod_{i=1}^R (1 - \alpha_i T) }{ \prod_{j=1}^S (1 - \beta_j T)} \in \bb Q(\zeta_p)(T).
\]
Furthermore, writing this in reduced form ($\alpha_i \not= \beta_j$ for every $i$ and $j$):
\begin{enumerate}
\item[(a)] The reciprocal zeros and poles $\alpha_i$ and $\beta_j$ are algebraic integers. For each reciprocal pole $\beta_j$ there is a reciprocal zero $\alpha_{k_j}$ and a positive integer $m_j$ such that $\beta_j = q^{m_j} \alpha_{k_j}$. 
\item[(b)] The degree $R - S$ of the $L$-function as a rational function is bounded as follows. If $\tilde s < s$ then $R = S$, else if $\tilde s = s$ then
\[
0 \leq R - S \leq s! \> vol(\Gamma) \c L N.
\]
\item[(c)] The total degree $R + S$ of the $L$-functions is bounded above by
\[
R+S \leq \c L N \cdot \tilde s! \> vol(\Gamma) \cdot 2^{ \tilde s  + (1 + \frac{1}{\tilde s}) n | \c L | } (1 + 2^{1 + \frac{1}{\tilde s}})^s.
\]
\end{enumerate}
\end{theorem}

Next, we consider the $L$-function defined over affine $s$-space. To this end, we assume $\bar P(x, t) \in \bb F_q[x_1^\pm, \ldots, n_n^\pm, t_1, \ldots, t_s]$. Set $M(\Gamma) := \bb Z^s \cap Cone(\Gamma)$, where $Cone(\Gamma)$ is the union of all rays from the origin through $\Gamma$. With $w_\Gamma$ the polyhedral weight function defined by $\Gamma$ in $\bb R^s$, define
\[
w(\Gamma) := \min\{ w_\Gamma(u) \mid u \in M(\Gamma) \cap \bb Z_{\geq 1}^s \}.
\]
Let $A \subset \{1, 2, \ldots, s\}$. Let $\bar G_A$ be the polynomial obtained from $\bar G$ by setting $t_i = 0$ for each $i \in A$. In precisely the same manner as $\Gamma$, let $\Gamma_A$ be the relative polytope of $\bar G_A$ and define its volume $vol(\Gamma_A)$ with respect to Haar measure normalized so that a fundamental domain of the integer lattice in the smallest subspace containing $\Gamma_A$ has unit volume.

\begin{theorem}\label{T: IntroMainAffine}
Suppose $\bar G \in \bb F_q[x_1^\pm, \ldots, x_n^\pm, t_1, \ldots, t_s]$. For each linear algebra operation $\c L$, the $L$-function $L(\c L \c A, \bb A^s / \bb F_q, T)$ is a rational function:
\[
L(\c L \c A, \bb A^s / \bb F_q, T)^{(-1)^{s+1}} = \frac{ \prod_{i=1}^{R} (1 - \alpha_i T)}{ \prod_{j=1}^{S}(1 - \beta_j T)} \in \bb Q(\zeta_p)(T).
\]
Writing this in reduced form ($\alpha_i \not= \beta_j$ for every $i$ and $j$):
\begin{enumerate}
\item[(a)] The reciprocal zeros and poles satisfy
\[
ord_q(\alpha_i), ord_q(\beta_j) \geq w(\Gamma).
\]
\item[(b)] The degree is bounded by
\[
- \sum_{\substack{A \subset \{1, 2, \ldots, s\} \\ |A| odd}} (s - |A|)! \> vol(\Gamma_A) \leq R - S \leq \sum_{\substack{A \subset \{1, 2, \ldots, s\} \\ |A| even}} (s - |A|)! \> vol(\Gamma_A).
\]
\item[(c)] The total degree is bounded by
\[
R + S \leq 2^{ \tilde s  + (1 + \frac{1}{\tilde s}) n | \c L | } \> 6^s \c L N \cdot \tilde s! \> vol(\Gamma)
\]
\end{enumerate}
\end{theorem}

Lastly, we look at the case of an affine family over an affine base. Suppose $\bar G$ now is a polynomial in $\bb F_q[x_1, \ldots, x_n, t_1, \ldots, t_s]$, and $\bar f$ is convenient and nondegenerate. For each $\lambda \in |\bb A^s / \bb F_q|$, let $\tilde{\c A}(\lambda) = \{\pi_i(\lambda)\}_{i=1}^{\tilde N}$ be the set of reciprocal zeros of $L(\bar G_\lambda, \Theta, \bb A^n / \bb F_q(\lambda), T)^{(-1)^{n+1}}$. Define $w(\Delta)$ for $\Delta = \Delta_\infty(\bar f)$ in a similar way to that of $w(\Gamma)$, so that $w(\Delta) = \min \{ w(\gamma) \mid \gamma \in M(\Delta) \cap \bb Z^n_{\geq 1} \}$.

\begin{theorem}
Suppose the conditions of the previous paragraph. For each linear algebra operation $\c L$, the $L$-function $L(\c L \tilde{\c A}, \bb A^s / \bb F_q, T)$ is a rational function:
\[
L(\c L \tilde{\c A}, \bb A^s / \bb F_q, T)^{(-1)^{s+1}} = \frac{ \prod(1 - \alpha_i T)}{ \prod(1 - \beta_j T)} \in \bb Q(\zeta_p)(T).
\]
Writing this in reduced form, then $ord_q(\alpha_i)$ and $ord_q(\beta_j) \geq w(\Gamma) + w(\Delta) \c L \tilde N$.
\end{theorem}

We note that the upper bound on the degree in Theorem \ref{T: IntroMain}{\it (b)} and the lower bound on the $p$-adic order of the roots in Theorem \ref{T: IntroMainAffine}{\it (a)} are sharp in the sense that there are examples where the bounds are obtained (e.g. \cite{Haessig-$L$-functionsofsymmetric-} and \cite{HaesRoja-$L$-functionsofsymmetric-}).

When we work with proper subsets $\c A_0 \subset \c A$, we in general do not expect rationality of the $L$-function, for example in the case when $\c A_0$ is the unit root family. As mentioned earlier, Wan's proof of Dwork's conjecture uses a $p$-adic limiting argument of $L$-functions associated to Adams operations. Since Adams operations may be viewed as a virtual linear algebra operation, we may apply the lower bound in Theorem \ref{T: IntroMainAffine}(a) to this sequence of $L$-functions to obtain a similar result for the unit root $L$-function (\ref{E: DworkUnit}). This is discussed in Sections \ref{SS: unit} and \ref{SS: Unitfam}.

We view the main contributions of the present study to be the discovery of the role played by the relative polytope $\Gamma$ for very general families of nondegenerate toric exponential sum. Previous $p$-adic studies have mainly been one parameter families in which the parameter appears linearly. In the present study we remove these restrictions. We are able nevertheless to compute relative cohomology, and the relative polytope provides a sufficiently good weight function so that the general results obtained are sharp in cases where the $L$-functions have previously been computed. In other earlier work (see \cite{AdolpSperb-Newtonpolyhedraand-1987} and \cite{AdolpSperb-ExponentialSumsand-1989}), understanding the weight function was an essential step in enabling the calculation of $p$-adic cohomology. It is our hope that the relative polytope provides a similar key step here. In a future article, we intend to treat families of Kloosterman-like sums, including the calculation of the relevant $p$-adic cohomology. It should be noted that while our main application has been to families of toric nondegenerate exponential sums these results have broader application to $\sigma$-modules with polyhedral growth, the content of which is in Section \ref{S: sigma-modules}.

{\bf Acknowledgment:} We thank Nick Katz for providing the $\ell$-adic proof of an upper bound for the $p$-adic order of eigenvalues of Frobenius used in Section \ref{S: Toric Family}.

\section{$L$-functions}\label{S: sigma-modules}

Let $\Gamma$ denote a fixed polytope in $\bb R^s$ with rational vertices which contains the origin, perhaps on its boundary. Let $\tilde s$ denote the dimension of the smallest linear space in $\bb R^s$ containing $\Gamma$. We will assume $\tilde s \geq 1$. Let $Cone(\Gamma)$ be the union of all rays from the origin through $\Gamma$ and set $M(\Gamma) := \bb Z^s \cap Cone(\Gamma)$. For each $u \in M$, we define the weight $w(u)$ of $u$ as the smallest nonnegative real number such that $u$ is in the dilation $w(u) \Gamma$. Since the values of $w$ on $M(\Gamma)$ may be described using rational linear forms coming from a finite number of top dimensional faces of the polytope, there exists a positive integer $D = D(\Gamma)$ such that $w(M(\Gamma)) \subset (1/D)\bb Z_{\geq 0}$. With $q = p^a$, let $\bb Q_q$ denote the unramified extension of $\bb Q_p$ of degree $a$, and $\bb Z_q$ its ring of integers. 

With an eye toward obtaining $p$-adic estimates for the Frobenius below, we fix $\pi$, a zero of the series $\sum_{j=0}^\infty y^{p^j} / p^j$ having $ord_p \pi = 1/(p-1)$. We are most interested in the extension $\bb Q_q(\zeta_p)$ of $\bb Q_p(\zeta_p)$. These fields have ring of integers respectively $\bb Z_q[\pi]$ and $\bb Z_p[\pi]$. At various points in the exposition we will want to take totally ramified extensions of $\bb Q_q(\zeta_p)$ and $\bb Q_p(\zeta_p)$. We will accomplish this by adjoining an appropriate root $\tilde \pi$ of $\pi$, say $\tilde \pi = \pi^{1/\tilde D}$ for some positive integer $\tilde D$.  The Frobenius automorphism $\sigma$ of $Gal(\bb Q_q / \bb Q_p)$ is extended to $Gal(\bb Q_q(\tilde \pi) / \bb Q_p(\tilde \pi))$ by setting $\sigma(\tilde \pi) = \tilde \pi$. Let us take then $K = \bb Q_q(\tilde \pi)$ and assume the ramification index of $K / \bb Q_p$ is $e$.

Denote by $\bb Z_q[\tilde \pi]$ the ring of integers of $K$.  Using the multi-index notation $t^u := t_1^{u_1} \cdots t_s^{u_s}$, where $u = (u_1, \ldots, u_s)$, we define the overconvergent power series ring $\c O^\dag_\Gamma$ as
\[
\c O^\dag_{\Gamma} := \left\{ \sum_{u \in M(\Gamma)} a_u t^u \mid a_u \in \bb Z_q[\tilde \pi], \liminf_{w(u) \rightarrow \infty} \frac{ ord_\pi( a_u) }{w(u)} > 0 \right\}.
\]

Let $A := ( A_{i,j} )_{i, j \in I}$ be a square (possibly infinite) matrix with entries in $\c O^\dag_{\Gamma}$. We assume the index set is countable (or finite) and we take it to be $I = \{0, 1, 2, \ldots \}$. We will assume that $A$ is \emph{nuclear}, meaning here its columns tend to zero $p$-adically (i.e. $\sup_i |A_{i,j}| \rightarrow 0$ as $j \rightarrow \infty$). When $A$ is a finite square matrix it is automatically nuclear since the nuclear condition is vacuous. Let $\hat t \in (\overline{\bb Q}_p^*)^{s}$ be the Teichm\"uller lifting of a point $\bar t \in  (\overline{\bb F}_q^*)^{s}$, and let $deg(\bar t) := [ \bb F_q(\bar t) : \bb F_q]$. We extend $\sigma$ to an automorphism of $\c O^\dag_\Gamma$ by acting on the coefficients of the power series. Define the matrix
\[
B(t) := A^{\sigma^{a-1}}(t^{p^{a-1}}) \cdots A^\sigma(t^p) A(t).
\]
Following Dwork \cite{Dwork-NormalizedPeriodsII}, we may associate to $B$ the $L$-function
\begin{equation}\label{E: LBdef}
L(B, \bb G_m^s / \bb F_q, T) := \prod_{\bar t \in |\bb G_m^s / \bb F_q|} \frac{1}{det(1 - B(\hat t^{q^{deg(\bar t)-1}}) \cdots B(\hat t^q) B(\hat t)  T^{deg(\bar t)})} \in 1 + T \bb Z_q[\pi][[T]],
\end{equation}
where the product runs over all closed points of the algebraic torus $\bb G_m^s$ over $\bb F_q$. This may be generalized so that the product runs over closed points of an algebraic variety; of particular interest is when the polytope $\Gamma$ lies within the first quadrant $\bb R^s_{\geq 0}$ so that we may define the $L$-function over affine $s$-space $\bb A^s$.

As it stands, $L(B, \bb G_m^s / \bb F_q, T)$ is not in general a rational function. In fact, it is not even $p$-adic meromorphic in general  (see \cite[Theorem 1.2]{Wan-p-adic-representation}). However, we will either insist on its rationality or assume a uniform overconvergent condition (\ref{E: uniform overconvergent}) which will guarantee $p$-adic meromorphy of the $L$-function.

It is useful to measure the nuclear condition on $A$ as follows. For each integer $j \geq 0$, let $d_j$ be the smallest nonnegative integer such that $A_{i,j'} \equiv 0$ mod($\tilde \pi^j$) for all $j' \geq d_j$, where we say $A_{ij} \equiv 0$ mod($\tilde \pi^j$) if $\tilde \pi^j$ divides every coefficient of the series $A_{ij} = A_{ij}(t)$. That is, $A_{i, d_j}, A_{i, d_j +1}, A_{i, d_j + 2}, \ldots$ is divisible by $\tilde \pi^j$. Note that $d_0 = 0$. Define $h_j := d_{j+1} - d_j$. This means that the first $h_0$ columns of $A$ are divisible by at least $\tilde \pi^0 = 1$, the next $h_1$ columns are divisible by at least $\tilde \pi$, the next $h_2$ columns by $\tilde \pi^2$, and so forth. Next, we need a function which will provide us with how divisible the $j$-th column is according to these $h_i$. Set $s(0) := 0$, and for each integer $j \geq 1$, let $s(j)$ denote the largest integer such that $j \geq d_{s(j)}$. That is, $s(j) = \ell$, where $\ell$ is the minimum integer such that $j \in [0, h_0 + h_1 + \cdots + h_\ell]$. This means that column $j$ of the matrix $A$ is divisible by at least $\tilde \pi^{s(j)}$.

Now let $A(t)$ be a square matrix indexed by $I$ (possibly infinite) with entries series in $\c O_\Gamma^\dag$ with coefficients in $K$. Let $b \in \bb R_{>0}$, $\rho \in \bb R$. Dwork defined spaces
\begin{align*}
L(b; \rho) &:= \left\{ \sum_{u \in M(\Gamma)} c_u t^u \mid c_u \in \bb Q_q(\tilde \pi), ord_p(c_u) \geq b w(u) + \rho \right\} \\
L(b) &:= \bigcup_{\rho \in \bb R} L(b; \rho).
\end{align*}
Write $A(t) = \sum_{u \in M(\Gamma)} a_u t^u$ where $a_u = (a_u(i, j) )_{i,j \in I}$ is a matrix with coefficients in $\bb Z_q[\tilde \pi]$. When $I$ is infinite we shall assume that
\begin{equation}\label{E: uniform overconvergent}
\liminf_{i, j \in I \text{ and } w(u) \rightarrow \infty} \frac{ ord_p( a_u(i, j) \tilde \pi^{-s(j)} ) }{w(u)} > 0,
\end{equation}
a {\it uniform overconvergent} condition. We note that $w(u)$ provides the polyhedral growth of the coefficients $a_u(i, j)$ with respect to $t^u$ while $\tilde \pi^{s(j)}$ describes the weight placed on each column index $j$. When $A$ satisfies (\ref{E: uniform overconvergent}), including the case when $I$ is finite, then there exists a positive real number $b$ in which all entries of $A$ lie within the Dwork space $L(b)$. Furthermore, there exists a real number $\rho$ such that $A_{i,j} \in L(b; \frac{s(j)}{e} + \rho)$ for all $i, j \in I$. In this case, we may define the matrix $A' := p^{-\rho} A$ whose entries $A_{i,j}' \in L(b; s(j)/e)$ for all $i, j \in I$, and
\[
L(A', \bb G_m^s / \bb F_q, T) = L(p^{-\rho} A, \bb G_m^s/ \bb F_q, T) = L( A, \bb G_m^s/ \bb F_q, p^{-\rho} T).
\]
Thus, we may at various points throughout the paper assume that $A$ has been normalized such that $\rho = 0$, or equivalently, $h_0 \not= 0 $.

Write $B(t) = \sum_{u \in M(\Gamma)} b_u t^u$. Define $F_B$ as the matrix $(b_{qu - v})_{(u,v)}$, where $u$ and $v$ run over $M(\Gamma)$, and we set $b_{qu - v} = 0$ if $qu - v \not\in M(\Gamma)$. That is, the $(u,v)$ block entry of $F_B$ is the matrix $b_{qu - v}$. Note, even if $B$ is a finite dimensional matrix, $F_B$ is infinite dimensional. Dwork's trace formula \cite[Lemma 4.1]{Wan-p-adic-representation} says
\begin{align}\label{E: Dworktrace}
L(B, \bb G_m^s / \bb F_q, T)^{(-1)^{s+1}} &= \prod_{i=0}^s det(1 - q^i F_B T)^{(-1)^i \binom{s}{i}} \notag  \\
&= det(1 - F_B T)^{\delta^s},
\end{align}
where $\delta$ sends an arbitrary function $g(T)$ to the quotient $g(T) / g(qT)$. Since the entries $a_u$ satisfy (\ref{E: uniform overconvergent}), so do the entries $b_u$; in particular, if $A_{i,j} \in L(b)$ then $B_{i,j} \in L(b/p^{a-1})$. Consequently, the Fredholm determinant $det(1 - q^i F_B T)$ is $p$-adic entire by \cite[Proposition 3.6]{Wan-p-adic-representation}, and thus $L(B, \bb G_m^s / \bb F_q, T)$ is $p$-adic meromorphic on $\bb C_p$ by (\ref{E: Dworktrace}).

\subsection{Main theorems on general $L$-functions}\label{SS: Main theorems}

In this section we prove two main results about $L$-functions of overconvergent matrices defined over $\c O_\Gamma^\dag$. We recall the following definitions and associated data from the previous section:
\begin{align*}
& \Gamma, \text{ a rational polytope in $\bb R^s$ with volume $vol(\Gamma)$} \\
& \tilde s, \text{ the dimension of the smallest linear space in $\bb R^s$ containing $\Gamma$} \\
& Cone(\Gamma), \text{ the cone in $\bb R^s$ over $\Gamma$} \\
& M(\Gamma) := \bb Z^s \cap Cone(\Gamma) \\
& K = \bb Q_q(\tilde \pi), \text{ with $ord_p(\tilde \pi) = 1/e$} \\
& A = (A_{i,j})_{i,j \in I}, \text{ a matrix with entries satisfying } A_{i,j} \in L(b; s(j)/e) \text{ for some fixed $b \in \bb Q_{>0}$} \\
& B, \text{ a matrix defined by }B(t) := A^{\sigma^{a-1}}(t^{p^{a-1}}) \cdots A^\sigma(t^p) A(t) \\
& w(\Gamma) := \min\{ w(u) \mid u \in M(\Gamma) \cap \bb Z_{>0}^s \} \\
& ord_p(A) := \min\{ s(i)/e \mid i \in I \}.
\end{align*}

It is convenient to assume that the field $K$ is sufficiently ramified so that the denominator of $b$ and $D(\Gamma)$ divide $e$. This does not for example change $ord_p(A)$. Suppose $A$, and hence $B$, are $N \times N$ matrices and write $dim(B) := N$. Suppose $L(B, \bb G_m^s / \bb F_q, T)$ is a rational function:
\begin{equation}\label{E: Lfunfrac}
L(B, \bb G_m^s / \bb F_q, T)^{(-1)^{s+1}} = \frac{\prod_{i=1}^R (1 - \alpha_i T)}{\prod_{j=1}^S (1 - \beta_j T)}
\end{equation}
written in reduced form (i.e. $\alpha_i \not= \beta_j$ for every $i$ and $j$). Let $Z := \{ \alpha_1, \ldots, \alpha_R \}$ be the multiset of reciprocal zeros and $P := \{ \beta_1, \ldots, \beta_S \}$ the multiset of reciprocal poles of (\ref{E: Lfunfrac}). Two elements $\gamma_1, \gamma_2 \in Z \cup P$ are said to be \emph{$q$-related} if there is an integer $\tau$ such that $\gamma_1 = q^\tau \gamma_2$. This defines an equivalence relation on $Z \cup P$. Let $\c E := \{ E_1, \ldots, E_\ell \}$ be the set of equivalence classes. Focussing on a fixed $E \in \c E$, define
\[
L_E(B, \bb G_m^s / \bb F_q, T)^{(-1)^{s+1}} := \frac{\prod_{ \alpha \in Z \cap E}  (1 - \alpha T)}{\prod_{\beta \in P \cap E} (1 - \beta T)},
\]
and if $R_E$ denotes the cardinality of $Z \cap E$ and $S_E$ the cardinality of $P \cap E$, then we may order the reciprocal zeros and poles giving
\[
L_E(B, \bb G_m^s / \bb F_q, T)^{(-1)^{s+1}} = \frac{\prod_{i=1}^{R_E}  (1 - \alpha_i T)}{\prod_{j=1}^{S_E} (1 - \beta_j T)}.
\]
Clearly,
\[
L(B, \bb G_m^s / \bb F_q, T) = \prod_{E \in \c E} L_E(B, \bb G_m^s / \bb F_q, T).
\]

\begin{lemma}\label{L: BombieriInj}
Under the conditions assumed above, then
\begin{enumerate}
\item for each $E \in \c E$, $R_E - S_E \geq 0$. (Consequently, the degree $R - S$ of $L(B, \bb G_m^s / \bb F_q, T)^{(-1)^{s+1}}$ as a rational function is nonnegative.)
\item For each $E \in \c E$, there is a choice $\gamma_E \in \{\alpha_1, \ldots, \alpha_{R_E} \}$ such that there are nonnegative integers $\{ m_i \}_{i=1}^{R_E}$ and strictly positive integers $\{ n_j \}_{j=1}^{S_E}$ with
\[
\alpha_i = q^{m_i} \gamma_E \quad \text{and} \quad \beta_j = q^{n_j} \gamma_E.
\]
\end{enumerate}
\end{lemma}

\begin{proof}
The following proof illustrates Bombieri's method \cite[Section 4, p.83]{Bombieri-exponentialsumsin-1966}. For $\gamma \in Z \cup P$ we write
\[
H(\gamma) := \prod_{m=0}^\infty (1 - q^m \gamma T)^{c(m)}
\]
where $c(m) := \binom{m+s-1}{s-1}$. The $\delta$-structure (\ref{E: Dworktrace}) implies
\[
det(1 - F_B T) = \frac{\prod_{\alpha \in Z} H(\alpha)}{\prod_{\beta \in P} H(\beta)},
\]
and this is $p$-adically entire as noted above. Note that for two elements $\gamma_1, \gamma_2 \in Z \cup P$, $\gamma_1$ and $\gamma_2$ are $q$-related if and only if $H(\gamma_1)$ and $H(\gamma_2)$ have a factor in common. As a consequence, if we write 
\[
D_E(T) := \frac{\prod_{\alpha \in Z \cap E} H(\alpha)}{\prod_{\beta \in P \cap E} H(\beta)},
\]
then
\[
det(1 - F_BT) = \prod_{E \in \c E} D_E(T)
\]
and each $D_E(T)$ is $p$-adically entire. If we fix $\beta_1 \in E \cap P$ so that $\beta_1 = q^{m_j} \beta_j$ with $m_j \geq 0$ for $j = 2, \ldots, S_E$, then we also have $\beta_1 = q^{t_i} \alpha_i$ with $t_i$ a non-zero integer for $i = 1, \ldots, S_E$. The precise divisibility of $(1 - q^m \beta_1 T)$ in $\prod_{i=1}^{R_E} H(\alpha_i)$,  for $m$ sufficiently large, is $\sum_{i=1}^{R_E} c(m + t_i)$ which grows with $m$ like a polynomial of the form $R_E \frac{ m^{s-1}}{(s-1)!} + (\text{lower order terms})$. Similarly, the precise divisibility of $(1- q^m \beta_1 T)$ in $\prod_{j=1}^{S_E} H(\beta_j)$  is $c(m) + \sum_{j=2}^{S_E} c(m+m_j)$ which grows like a polynomial of the form  $S_E \frac{m^{s-1}}{(s-1)!} + (\text{lower order terms})$. Since $D_E(T)$ is $p$-adically entire, any factor $(1 - q^m \beta_1 T)$ must have nonnegative exponent in $\prod_{i=1}^{R_E} H(\alpha_i) / \prod_{j=1}^{S_E} H(\beta_j)$ so that for $m$ sufficiently large, we must have $R_E - S_E \geq 0$. This completes the proof of the first part of the theorem. 

The proof of the second part is simpler. Note that $D_E(T)$ is entire so that each factor $(1- \beta T)$ must divide $\prod_{i=1}^{R_E} H(\alpha_i)$, as a consequence $\beta_j = q^{m_j} \alpha_{k(j)}$ for some $m_j \geq 1 $ and $1 \leq k(j) \leq R_E$. The result then follows by choosing $\gamma_E$ among $\{ \alpha_1, \ldots, \alpha_{R_E} \}$ so that $\alpha_i = q^{m_i} \gamma_E$ for nonnegative integers $\{m_i\}_{i=1}^{R_E}$. 
\end{proof}

We note that it is an interesting question to determine an upper bound for the degree of $L_E(B, \bb G_m^s / \bb F_q, T)$, as well as an estimate for the number of equivalence classes in $\c E$.

\begin{theorem}\label{T: MainDworkTheorem}
Suppose $A$, and hence $B$, are $N \times N$ matrices and write $dim(B) := N$. In parts (a) through (c) below, we assume that $L(B, \bb G_m^s / \bb F_q, T)$ is a rational function:
\[
L(B, \bb G_m^s / \bb F_q, T)^{(-1)^{s+1}} = \frac{\prod_{i=1}^R (1 - \alpha_i T)}{\prod_{j=1}^S (1 - \beta_j T)}
\]
written in reduced form (i.e. $\alpha_i \not= \beta_j$ for every $i$ and $j$). In part (d), we assume condition (\ref{E: uniform overconvergent}) so that this $L$-function is $p$-adic meromorphic. Then,
\begin{enumerate}
\item[a)] The reciprocal zeros and poles $\alpha_i$ and $\beta_j$ are algebraic integers. Also, for each reciprocal pole $\beta_j$ there is a reciprocal zero $\alpha_{k_j}$ and a positive integer $m_j$ such that $\beta_j = q^{m_j} \alpha_{k_j}$. 
\item[b)] If $\tilde s < s$ then $R = S$, else if $\tilde s = s$ then
\[
0 \leq R - S \leq \left( \frac{1}{b(p-1)} \right)^s \cdot s! \> vol(\Gamma) \cdot dim(B)
\]
\item[c)] Let $k$ be the smallest positive integer such that $ord_q \alpha_i$ and $ord_q \beta_j \leq k$ for all $i$ and $j$. Then, with $\rho := \min\{s, k\}$, we have
\[
R + S \leq
dim(B) \cdot \tilde s! \> vol(\Gamma) \cdot 2^{\tilde s + \frac{1}{b(p-1)}(1 + \frac{1}{\tilde s})(k - \rho)} ( 1 + 2^{\frac{1}{b(p-1)}(1 + \frac{1}{\tilde s})})^\rho.
\]
\item[d)] Suppose $\Gamma \subset \bb R_{\geq 0}^s$ and $b(p-1) \leq 1$. Write
\[
L(B, \bb A^s / \bb F_q, T) = \frac{ \prod (1 - \alpha_i T)}{\prod (1 - \beta_j T)}.
\]
Here, $A$, and hence $B$, may be infinite dimensional, in which case $\alpha_i$ and $\beta_j \rightarrow 0$ $p$-adically. Then $ord_q \alpha_i$ and $ord_q \beta_j \geq b(p-1) w(\Gamma) + ord_p(A)$ for all  $i$ and $j$. A lower bound is also given when $b(p-1) > 1$.
\end{enumerate}
\end{theorem}

\begin{proof}[Proof of Theorem \ref{T: MainDworkTheorem}, part (a)]
That the reciprocal roots are algebraic integers follows directly from Dwork's argument found in Bombieri's paper \cite[Section 4, p.82]{Bombieri-exponentialsumsin-1966}. The second part follows from Lemma \ref{L: BombieriInj}.
\end{proof}

The proof of parts {\it (b)}, {\it (c)}, and {\it (d)} will require various lemmas. We begin by obtaining a lower bound on the $q$-adic Newton polygon of $det(1 - F_B T)$. With this purpose in mind, we define a $p$-adic Banach space $C(b, I)$ over the field $K_0 := \bb Q_p(\tilde \pi)$, having orthonormal basis $\{ \gamma_b^{w(u)} t^u e_i \}_{u \in M(\Gamma), i \in I}$ where $ord_p(\gamma_b) = b$. Thus
\[
C(b, I) := \left\{ \xi = \sum_{i \in I, u \in M(\Gamma)} c(u, i) \gamma_b^{w(u)} t^u e_i \mid c(u, i) \in K_0, c(u, i) \rightarrow 0 \text{ as } (u, i) \rightarrow \infty \right\}.
\]
As usual, the norm on $C(b, I)$ is given by $| \xi | := \sup\{ |c(u, i)| : (u, i) \in M(\Gamma) \times I \}$. The map $\psi_p \circ A(t)$ acts on $\xi \in C(b, I)$ via (recall the notation $a_v(i, j)$ from the sentence before equation (\ref{E: uniform overconvergent}))
\[
(\psi_p \circ A(t)) \xi = \sum_{i \in I, \ell \in M(\Gamma)} \left( \sum_{j \in I, u + v = p \ell} a_v(i, j) c(u, j) \gamma_b^{w(u) - w(\ell)} \right) \gamma_b^{w(\ell)} t^\ell e_i.
\]
Then $\Phi_A := \sigma^{-1} \circ \psi_p \circ A(t)$ is a completely continuous endomorphism of $C(b, I)$ over $K_0$, semi-linear with respect to $\sigma^{-1}$ over $K$. The map $\Phi_B := \psi_q \circ B(t)$ is a completely continuous endomorphism of $C(b, I)$ over $K$ satisfying $\Phi_B = \Phi_A^a$. Let $\c B := \{ t^u e_i \mid (u, i) \in M(\Gamma) \times I \}$. Then the matrix of $\Phi_B$ computed with respect to the basis $\c B$ is $F_B$. As is well-known, a completely continuous endomorphism has a well-defined Fredholm determinant $det_{K}(1 - \Phi_B T)$. This may be computed using the matrix of $\Phi_B$ with respect to $\c B$.

Let $\{\eta_1, \ldots, \eta_a\} \subset \bb Z_q$ be a lifting of a basis of $\bb F_q$ over $\bb F_p$. Then $\{ \eta_j\}_{j=1}^a$ is a basis of $\bb Q_q$ over $\bb Q_p$ such that for every $\xi \in \bb Q_q$, writing $\xi = \sum_{j=1}^a \xi_j \eta_j$ with $\xi_j \in \bb Q_p$, we have $ord_p(\xi) = \min \{ ord_p(\xi_j) \}$.  The Fredholm determinant of $\Phi_A$ as a completely continuous endomorphism of the $K_0$-space $C(b, I)$ may be calculated from the matrix of $\Phi_A$ with respect to the basis $\c B' := \{ \eta_j t^u e_i : 1 \leq j \leq a, u \in M(\Gamma), i \in I\}$. The relation the between $det_{K_0}(1 - \Phi_A T)$ and $det_{K}(1 - \Phi_B T)$ is given by the following lemma.

\begin{lemma}(cf. \cite[Lemma 7.1]{Dwork-zetafunctionof-1964})\label{L: NPrelation1}
Taking the graph of the $p$-adic Newton polygon of $det_{K_0}(1 - \Phi_A T)$ and rescaling the abscissa and ordinate by $1/a$ yields the $q$-adic Newton polygon of $det_{K}(1 - \Phi_B T)$.
\end{lemma}

\begin{proof}
For convenience, write $G(T) := det_{K}(1 - \Phi_B T)$. Then
\begin{align*}
det_{K_0}(1 - \Phi_B T) &= Norm_{K / K_0} G(T) \\
&= G^{\sigma^{a-1}}(T) \cdots G^\sigma(T) G(T).
\end{align*}
Next,
\begin{align*}
det_{K_0}(1 - \Phi_B T^a) &= det_{K_0}(1 - \Phi_A^a T^a) \\
&= \prod_{\zeta^a = 1} det_{K_0}(1 - \zeta \Phi_A T).
\end{align*}
Thus,
\begin{equation}\label{E: Grelation}
G^{\sigma^{a-1}}(T^a) \cdots G^\sigma(T^a) G(T^a) = \prod_{\zeta^a = 1} det_{K_0}(1 - \zeta \Phi_A T).
\end{equation}
Let $N$ denote the number of reciprocal roots of $G(T)$ with slope $ord_q = m$, which means $ord_p = am$. Then the lefthand-side of (\ref{E: Grelation}) has $(aN)^a$ number of reciprocal roots of $ord_p = m$. Consequently, since $\zeta$ does not affect the Newton polygon of $det_{K_0}(1 - \zeta \Phi_A T)$, we see that $det_{K_0}(1 - \Phi_A T)$ has $a N$ reciprocal zeros of $ord_p = m$. The lemma follows.
\end{proof}

We proceed now to an estimate for the $p$-adic Newton polygon of $det_{K_0}(1 - \Phi_A T)$. Recall,  $A(t) = (A_{i,j})$ satisfies $A_{i,j} \in L(b; s(j)/e)$ for every $j$, with $b$ a positive rational number. Let $d$ be the smallest positive integer such that $b(p-1)w(u) + s(i)/e \in \frac{1}{d} \bb Z$ for all $u \in M(\Gamma)$ and $i \in I$.  Note, this means $db(p-1)$ and $d s(i)/e$ are nonnegative integers for all $i \in I$. Define
\[
W(j) := \# \left\{ (u,i) \in M(\Gamma) \times I \mid b(p-1) w(u) + \frac{s(i)}{e} = \frac{j}{d} \right\}.
\]

\begin{lemma}\label{L: NPrelation}
The $p$-adic Newton polygon of $det_{K_0}(1 - \Phi_A T)$ lies on or above the lower convex hull in $\bb R^2$ of the points
\[
(0,0) \quad \text{and} \quad \left( a \sum_{j=0}^n W(j), \frac{a}{d} \sum_{j=0}^n j W(j) \right) \quad n=0, 1, 2, \ldots.
\]
Consequently, from Lemma \ref{L: NPrelation1}, the $q$-adic Newton polygon of $det(1 - F_B T)$ lies on or above the lower convex hull in $\bb R^2$ of the points
\begin{equation}\label{E: NPlowerbound}
(0,0) \quad \text{and} \quad \left( \sum_{j=0}^n W(j), \frac{1}{d} \sum_{j=0}^n j W(j) \right) \quad n=0, 1, 2, \ldots.
\end{equation}
\end{lemma}

\begin{proof}

Write $A(t) = \sum_{u \in M(\Gamma)} a_u t^u$, where each $a_u = (a_u(i,j) )_{i, j \in I}$ is a matrix. For each basis element $e_i$, write $a_u e_i = \sum_{j \in I} a_u(i, j) e_j$ with $a_u(i, j) \in K$.  We now compute the matrix of $\Phi_A$ with respect to the basis $\c B'$. For $\eta_l t^u e_i \in \c B'$,
\begin{align*}
\Phi_A( \eta_l t^u e_i) &= \sigma^{-1} \circ \psi_p\left( A(t) \eta_l t^u e_i \right) \\
&= \sigma^{-1}(\eta_l) \cdot \psi_p\left( \sum_{v \in M(\Gamma)} \sum_{j \in I} a_v(i, j) t^{v+u} e_j \right) \\
&= \sigma^{-1}(\eta_l) \cdot \psi_p\left( \sum_{v \in M(\Gamma)} \left( \sum_{j \in I} a_v(i, j) e_j \right) t^{u+v} \right) \\
&= \sigma^{-1}(\eta_l) \cdot \psi_p\left( \sum_{r \in M(\Gamma)} \left( \sum_{j \in I} a_{r-u}(i, j) \right) e_j t^r \right) \\
&= \sigma^{-1}(\eta_l) \cdot \sum_{r \in M(\Gamma)} \sum_{j \in I} a_{pr-u}(i,j) e_j t^r.
\end{align*}
For each $i$ and $j$, write $a_u(i, j) = \sum_{k=1}^a a_u(i, j; k) \eta_k$, with $a_u(i, j; k) \in K_0$. Continuing the above calculation,
\[
\Phi_A( \eta_l t^u e_i) = \sigma^{-1}(\eta_l) \cdot \sum_{r \in M(\Gamma)} \sum_{j \in I} \sum_{k=1}^a a_{pr-u}(i, j; k) \eta_k t^r e_j.
\]
Next, for each $l$ and $k$, write $\sigma^{-1}(\eta_l) \eta_k = \sum_{m=1}^a b(l, k; m) \eta_m$ with $b(l, m; k) \in \bb Z_p$. Then, continuing the calculation
\[
\Phi_A( \eta_l t^u e_i) = \sum_{r \in M(\Gamma)} \sum_{j \in I} \sum_{k, m = 1}^a a_{pr-u}(i, j; k) \ b(l, k; m) \eta_m t^r e_j.
\]
Hence,
\[
\text{Matrix of $\Phi_A$ with respect to the basis $\c B'$ is }  \left( \sum_{k=1}^a a_{pr-u}(i, j; k)  \ b(l, k; m) \right)_{(l, u, i), (m, r, j)}.
\]
Set $d(l, u, i; m, r, j) := \sum_{k=1}^a a_{pr-u}(i, j; k) b(l, k; m)$. Then
\[
det_{K_0}(1 - \Phi_A T) = \sum_{m=0}^\infty c_m T^m
\]
where
\[
c_m := (-1)^m \sum \sum_{\tau \in S_m} sgn(\tau) \prod_{z=1}^m d(l^{(z)}, u^{(z)}, i^{(z)} ; l^{(\tau(z))}, u^{(\tau(z))}, i^{(\tau(z))})
\]
where the first sum runs over all $m$ number of triples $(l^{(1)}, u^{(1)}, i^{(1)}) , \ldots, (l^{(m)}, u^{(m)}, i^{(m)})$ of distinct elements of the set $\{1, 2, \ldots, a\} \times M(\Gamma) \times I$, and $S_m$ is the symmetric group on the letters $\{1, 2, \ldots, m\}$. Recall that $ord_p( a_u(i, j) ) \geq b w_\Gamma(u) + \frac{s(i)}{e}$, and so by construction of the basis $\{ \eta_j \}$, the same holds true for each $a_u(i, j; k)$. Since $ord_p( b(l, k; m) ) \geq 0$, we have
\begin{align*}
ord_p(c_m) &\geq \min_{\text{distinct } (l, u, i)} \min_{\tau \in S_m} \sum_{z=1}^m \left( b w_\Gamma(p u^{(\tau(z))} - u^{(z)}) + \frac{s(i^{\tau(z)})}{e} \right) \\
&\geq \min_{\text{distinct } (l, u, i)} \left\{ \sum_{z=1}^m \left( b(p-1) w_\Gamma(u^{(z)}) + \frac{s(i^{(z)})}{e} \right) \right\}.
\end{align*}
It follows that the $p$-adic Newton polygon of $det_{K_0}(1 - \Phi_A T)$ lies on or above the lower convex hull in $\bb R^2$ of the points
\[
(0,0) \quad \text{and} \quad \left( a \sum_{j=0}^n W(j), \frac{a}{d} \sum_{j=0}^n j W(j) \right) \quad n=0, 1, 2, \ldots.
\]\
\end{proof}

\begin{proof}[Proof of Theorem \ref{T: MainDworkTheorem}, part (b)]
Bombieri's argument \cite[Section IV]{Bombieri-exponentialsumsin-1966} demonstrates that using the lower bound (\ref{E: NPlowerbound}) and the Dwork trace formula (\ref{E: Dworktrace}), one may obtain the inequality
\[
\frac{deg \> L(B, \bb G_m^s, T)^{(-1)^{s+1}}}{(s+1)!} x^{s+1} + O(x^s) \leq \sum_{\frac{j}{d} \leq x} \left( x - \frac{j}{d} \right) W(j).
\]
As we show below, the righthand side may be asymptotically approximated by
\begin{equation}\label{E: asympEst}
\sum_{\frac{j}{d} \leq x} \left(x - \frac{j}{d} \right) W(j) = \frac{dim(B) \> vol(\Gamma)}{(\tilde s + 1)(b(p-1))^{\tilde s}} x^{\tilde s + 1} + O(x^{\tilde s}),
\end{equation}
which proves the result since $\tilde s \leq s$. The proof of this estimate is a modification of an argument of \cite[\S 4]{AdolpSperb-Newtonpolyhedraand-1987}, thus we will only provide the relevant parts. First, we note that we may write
\begin{equation}\label{E: sillysum}
W(j) = \sum_{i \in I} W_i(j)
\end{equation}
where
\[
W_i(j) := \# \{ u \in M(\Gamma) \mid b(p-1) w(u) + s(i) / e = j/d \}.
\]
For convenience, set $b' := b(p-1)$ and $c_i := s(i)/e$. Recall that $w(M(\Gamma)) \subset (1/D) \bb Z_{\geq 0}$. Now, if $u$ satisfies $b' w(u) + c_i = j/d$ then there exists a nonnegative integer $j'$ such that $w(u) = j'/D$ such that $b' j' / D + c_i = j/d$. Hence, for each fixed $i$,
\begin{align*}
\sum_{\frac{j}{d} \leq x} W_i(j) &= \sum_{\frac{j'}{D} \leq \frac{x-c_i}{b'}} \#\{ u \in M(\Gamma) \mid w(u) = j'/D \} \\
&= vol(\Gamma) \left( \frac{x}{b'} \right)^{\tilde s} + O(x^{\tilde s-1}),
\end{align*}
where we have used the argument in \cite[\S 4]{AdolpSperb-Newtonpolyhedraand-1987} for the second equality. A similar calculation gives
\begin{align*}
\sum_{\frac{j}{d} \leq x} \left( \frac{j}{d} \right) W_i(j) &= \sum_{\frac{j'}{D} \leq \frac{x-c_i}{b'}} \left( \frac{b' j'}{D} + c_i \right) \#\{ u \in M(\Gamma) \mid w(u) = j'/D \} \\
&= \frac{\tilde s \> vol(\Gamma)}{\tilde s + 1} \left( \frac{x}{b'} \right)^{\tilde s + 1} + O(x^{\tilde s}).
\end{align*}
Then, (\ref{E: asympEst}) follows by combining these estimates with (\ref{E: sillysum}).
\end{proof}

We now move on to an upper bound for the total degree. Like the degree, the proof is a modification of an argument of Bombieri's \cite{Bombieri-exponentialsumsin-1978}. However, we will follow the argument in \cite{AdolpSperb-Newtonpolyhedraand-1987} since it applies to growth conditions defined by polytopes.

\begin{proof}[Proof of Theorem \ref{T: MainDworkTheorem}, part (c)] In fact, we prove the slightly stronger inequality
\begin{equation}\label{E: slight}
R + S \leq
dim(B) \cdot \tilde s! \> vol(\Gamma) \cdot 2^{\tilde s - \frac{1}{\tilde s d b(p-1)} + \frac{1}{b(p-1)}(1 + \frac{1}{\tilde s})(k - \rho)} ( 1 + 2^{\frac{1}{b(p-1)}(1 + \frac{1}{\tilde s})})^\rho.
\end{equation}
Inequality (\ref{E: slight}) follows from an almost identical argument to that in \cite[\S 5]{AdolpSperb-Newtonpolyhedraand-1987} once one has the following rationality result of the Poincar\'e series $\sum_{N=0}^\infty W(N) T^N$:
\begin{equation}\label{E: PoincSer}
\sum_{N=0}^\infty W(N) T^N = \frac{Q(T)}{(1 - T^{db(p-1)})^{\tilde s}},
\end{equation}
where $Q(T)$ is a polynomial with nonnegative integer coefficients and special value $Q(1) = dim(B) \cdot {\tilde s}! \> vol(\Gamma)$.
We prove this as follows. Define
\[
W'(N) := \# \{ u \in M(\Gamma) \mid w(u) = N/D \}.
\]
The Poincar\'e series of this, by \cite[Lemma 5.1]{AdolpSperb-Newtonpolyhedraand-1987}, satisfies
\[
\sum_{N' = 0}^\infty W'(N') T^{N'} = \frac{P(T)}{(1 - T^D)^{\tilde s}}
\]
with $P(T)$ a polynomial in nonnegative integral coefficients, degree at most $\tilde s D$, and special value $P(1) = \tilde s! \> vol(\Gamma)$.
Using (\ref{E: sillysum}), observe that
\[
\sum_{N=0}^\infty W(N) T^N = \sum_{i \in I} \sum_{N=0}^\infty W_i(N) T^N.
\]
For convenience, set $b' := b(p-1)$ and $c_i := s(i) / e$. Notice that $W_i(N)$ is a positive integer if and only if there is a $u \in M(\Gamma)$ that satisfies $b' w(u) + c_i = N/d$. By defining the integer $N'$ by $w(u) = N'/D$, we see that $W_i(N) = W'(N')$ when $\frac{N}{d} = b' \cdot \frac{N'}{D} + c_i$. Thus,
\begin{align*}
\sum_{i \in I} \sum_{N=0}^\infty W_i(N) T^N &= \sum_{i \in I} \sum_{N' = 0}^\infty W'(N') T^{d( (b'N'/d) + c_i)} \\
&=  \frac{P(T) \sum_{i \in I} T^{dc_i}}{(1 - T^{db'})^{\tilde s}}.
\end{align*}
Note, by the discussion before Lemma \ref{L: NPrelation}, $db'$ and $dc_i$ are nonnegative integers. Denoting the numerator of this quotient by $Q(T)$ proves (\ref{E: PoincSer}).
\end{proof}

We remark that a more modern and accessible presentation of Ehrhart's theory on polytopes (e.g. rationality of the Poincar\'e series mentioned in the above proof) may be found in \cite[Chapter 3]{Beck}.

Assume now that $\Gamma$ lies entirely in $\bb R_{\geq 0}^s$ so that the $L$-function may be defined over affine $s$-space $\bb A^s$. Define
\[
S(B) := \sum_{\bar t \in \bb F_q^s} Tr( B(t) ),
\]
where $t$ is an $s$-tuple in $\bb Z_q^s$ whose coordinates are the Teichm\"uller representative of $\bar t$ (or zero).

\begin{theorem}\label{T: GenExpSumEst}
Suppose $\Gamma$ lies entirely in $\bb R_{\geq 0}^s$. Here $A$ may be infinite dimensional. Assume $b(p-1) \leq 1$. Then
\begin{equation}\label{E: CW}
ord_q \> S(B) \geq b(p-1) w(\Gamma) + ord_p(A).
\end{equation}
A lower bound is also given if $b(p-1) > 1$.
\end{theorem}

\begin{proof}
The underlying idea of the following goes back to Katz \cite{Katz-On_thm_of_Ax}. Write $B(t) = \sum_{u \in M(\Gamma)} b_u t^u$. For $C \subset \{1, 2, \ldots, s\}$, define the matrix $B^C$ by
\[
B^C := \sum_{\substack{ u \in M(\Gamma) \\ u_i > 0 \text{ if } i \in C }} b_u t^u.
\]
We use the following version of the Dwork trace formula \cite[Lemma 4.3]{Wan-p-adic-representation}:
\[
L(B, \bb A^s / \bb F_q, T) = \prod_{C \subset \{1, 2, \ldots, s\}} det(1 - q^{s - |C|} F_{B^C} T)^{(-1)^{s - |C|}}.
\]
Hence,
\[
S(B) = \sum_{C \subset \{1,2, \ldots, s\}} (-1)^{|C|} q^{s-|C|} Tr( F_{B^C} ),
\]
and since $q=p^a$, we have
\[
ord_q \> S(B)  \geq \min_{C \subset \{1, 2, \ldots, s\}} \left\{ s-|C| + ord_q( Tr( F_{B^C} )) \right\}.
\]
We now estimate $ord_q( Tr( F_{B^C} ))$. Let $E^C$ denote the set of $u \in  M(\Gamma)$ with $u_i > 0 $ for all $i \in C$. Define
\begin{equation}\label{E: WC}
W^C(j) := \# \left\{ (u, i) \in E^C \times I \mid b(p-1)w(u) + \frac{s(i)}{e} = \frac{j}{d} \right\}.
\end{equation}
Using a similar argument to Lemma \ref{L: NPrelation}, the $q$-adic Newton polygon of $det_{K}(1 - F_{B^C} T)$ lies on or above the points
\[
(0,0) \quad \text{and} \quad \left( \sum_{j = 0}^k W^C(j), \frac{1}{d} \sum_{j = 0}^k j W^C(j) \right) \qquad k = 0, 1, 2, \ldots
\]
Let $k_0$ be the first nonnegative integer such that $W^C(k_0) \not= 0$. Then the first point away from the origin in the above sequence is $(W^C(k_0), \frac{k_0}{d} W^C(k_0))$, and thus, since $det_{K}(1 - F_{B^C} T) = 1 - Tr( F_{B^C} ) T + O(T^2)$, we obtain the inequality
\[
ord_q( Tr( F_{B^C} ) ) \geq \frac{k_0}{d} = \min_{(u, i) \in E^C \times I} \left\{ b(p-1)w(u) + \frac{s(i)}{e} \right\}.
\]

Now, using \cite[Lemma 4.5]{AdolpSperb-$p$-adicestimatesexponential-1987} for the second inequality below, we see that
\begin{align*}
ord_q \> S(B) &\geq \min_{C \subset \{1, 2, \ldots, s\}} \left\{ s - |C| + \min_{(u, i) \in E^C \times I} \left\{ b(p-1)w(u) + \frac{s(i)}{e} \right\} \right\} \\
&=
\min_{C \subset \{1, 2, \ldots, s\}} \left\{ s - |C| + b(p-1) \min_{u \in E^C} \left\{ w(u) \right\} \right\} + ord_p(A) \\
&\geq
\min_{C \subset \{1, 2, \ldots, s\}} \left\{ s - |C| + b(p-1) \left( w(\Gamma) - (s - |C|) \right) \right\} + ord_p(A) \\
&=
\begin{cases}
b(p-1) w(\Gamma) + ord_p(A) & \text{if } b(p-1) \leq 1 \\
s(1-b(p-1)) + b(p-1) w(\Gamma) + ord_p(A) & \text{if } 1 < b(p-1),
\end{cases}
\end{align*}
taking $|C| = s$ in the former and $|C| = 0$ in the latter.
\end{proof}

\begin{proof}[Proof of Theorem \ref{T: MainDworkTheorem}, part (d)]
For convenience, set $\kappa := b(p-1)w(\Gamma) + ord_p(A)$. With
\[
S_k(B) := \sum_{\bar t \in (\bb F_{q^k})^s} Tr( B(t) ),
\]
it follows from Theorem \ref{T: GenExpSumEst} that $ord_q( S_k(B) ) \geq k \kappa$ for every $k$. Using the same idea as in \cite{Ax-Zeros_of_polynomials_over_finite_fields}, this is equivalent to every reciprocal zero and pole of the $L$-function of $B$ having $p$-adic order at least $\kappa$.
\end{proof}

\subsection{Unit root $\sigma$-modules}\label{SS: unit}

Using freely the terminology of \cite{Wan-Dwork'sconjectureunit-1999} and \cite{Wan-Higherrankcase-2000}, we will demonstrate that part {\it (d)} of Theorem \ref{T: MainDworkTheorem} may be applied to obtain a similar result for unit root $\sigma$-modules.  In Section \ref{SS: Unitfam}, we will apply the following result to a particular unit root $L$-function coming from the nondegenerate toric family in Section \ref{S: Toric Family}. Let $(M, \phi)$ be a finite rank, ordinary, nuclear $\sigma$-module. (Here, ordinary means that the fiber-by-fiber Newton polygon of $\phi$ equals the polygon defined as the lower convex hull of the points $( \sum_{i=0}^k h_i, \sum_{i=0}^k i h_i)$ in $\bb R^2$, where $h_i$ was defined in Section \ref{S: sigma-modules}.) We assume there exists a matrix $A(t)$ such that the matrix $B(t)$ of $\phi$ with respect to some orthonormal basis $\c B$ satisfies
\begin{equation}\label{E: Bfactor}
B(t) = A^{\sigma^{a-1}}(t^{p^{a-1}}) \cdots A^\sigma(t^p) A(t)
\end{equation}
with $A(t)$ satisfying (\ref{E: uniform overconvergent}). The latter condition means there exists $b > 0$ such that the entries of $A(t)$ all belong to $L(b)$.  Let $\phi_i$ be the $i$-th slope unit root $\sigma$-module coming from the Hodge-Newton decomposition of $\phi$. Wan's theorem \cite{Wan-Higherrankcase-2000} tells us that the $L$-function of $\phi_i$ is meromorphic:
\[
L(\phi_i, \bb A^s / \bb F_q, T) = \frac{\prod_{i=1}^\infty (1 - \alpha_i T)}{\prod_{j=1}^\infty (1 - \beta_j T)},
\]
with $\alpha_i, \beta_j \rightarrow 0$ as $i, j \rightarrow \infty$.

\begin{theorem}\label{T: unitRoot}
Assume the relative polytope (\ref{E: GammaRelPoly}) satisfies $\Gamma \subset \bb R_{\geq 0}$. Suppose $b(p-1) \leq 1$. Then for all $i$ and $j$, $ord_q \alpha_i$ and $ord_q \beta_j \geq b(p-1) w(\Gamma)$. A lower bound may also be given if $b(p-1) > 1$.
\end{theorem}

\begin{proof}
In the proof of Wan \cite[Theorem 6.7]{Wan-Higherrankcase-2000} it was shown that there exists a sequence of matrices $\{ B_j(t) \}$ with the following properties. There exists matrices $A_j(t)$ for each $j$ with entries in $L(b)$,  $b$ independent of $j$, such that
\[
B_j(t) = A_j^{\sigma^{a-1}}(t^{p^{a-1}}) \cdots A_j^\sigma(t^p) A(t),
\]
with
\[
ord_p(A_j) \geq 0 \text{ for every $j$} \quad \text{and} \quad \lim_{j \rightarrow \infty} ord_p(A_j) = \infty
\]
satisfying
\[
L(\phi_i, \bb A^s/\bb F_q, T) = \prod_{j = 1}^\infty L(B_j, \bb A^s / \bb F_q, T)^{\pm},
\]
where $\pm$ means the factor lies in either the numerator or the denominator. The result now follows by applying part {\it (d)} of Theorem \ref{T: MainDworkTheorem} to each $L$-function in the product.
\end{proof}

When the unit root $\sigma$-module is of rank one, then we may allow the matrix $A(t)$ to have possibly infinite dimension. This is useful in applications, as we will see in Section \ref{SS: Unitfam}.

\begin{theorem}\label{T: unitrankone}
Let $(M, \phi)$ be a possibly infinite rank nuclear $\sigma$-module, ordinary up to and including slope $j$. Suppose there exists $b > 0$ and a matrix $A(t)$ with entries in $L(b)$, $\Gamma \subset \bb R^s_{\geq 0}$,  such that the matrix $B(t)$ of $\phi$ with respect to some orthonormal basis satisfies (\ref{E: Bfactor}). Let $\phi_j$ be the unit root $\sigma$-module coming from the $j$-th slope in the Hodge-Newton decomposition of $\phi$. Suppose $\phi_j$ is of rank one. Then the same result as in Theorem \ref{T: unitRoot} holds for the unit root $L$-function $L(\phi_j, \bb A^s / \bb F_q, T)$ by the proof of \cite[Theorem 8.5]{Wan-Dwork'sconjectureunit-1999}.
\end{theorem}

\section{Families of nondegenerate toric sums}\label{S: Toric Family}

{\bf Notation.} Let $V$ be a finite subset of $\bb Q^n$ and define $\Delta(V)$ as the convex closure of $V \cup \{ 0 \}$ in $\bb R^n$. Let $Cone(V)$ be the union of all rays from the origin through $\Delta(V)$. Set $M(V) := \bb Z^n \cap Cone(V)$, a monoid. The monoid-algebra $R(V) := \bb F_q[M(V)]$ may be filtered using a (polyhedral) weight function as follows. For each $\mu \in M(V)$, let $w_V(\mu)$ be the smallest non-negative rational number such that $\mu \in w_V(\mu) \Delta(V)$. Then $w_V(M(V)) \subset \frac{1}{D(V)} \bb Z_{\geq 0}$ for some fixed positive integer $D(V)$. The following properties of $w_V$ hold.
\begin{enumerate}\label{E: weightproperties}
\item[(i)] $w_V(\mu) = 0$ if and only if $\mu = 0$.
\item[(ii)] $w_V(c \mu) = c w_V(\mu)$, if $c \geq 0$ and $\mu \in M(V)$.
\item[(iii)] $w_V(\mu + \nu) \leq w_V(\mu) + w_V(\nu)$ for all $\mu, \nu \in M(V)$.
\end{enumerate}
Furthermore, equality holds in (iii) if and only if $\mu$ and $\nu$ are cofacial with respect to the same closed face of $\Delta(V)$, i.e. the rays from 0 to $\mu$ and from 0 to $\nu$ intersect a common closed face of $\Delta(V)$.

The weight function $w_V$ imparts an increasing filtration to the ring $R(V)$ defined by
\[
\text{Fil}_i R(V) := \{ \bar g \in R(V) \mid w_V(\mu) \leq i \text{ for all } \mu \in \text{Supp}(\bar g) \}
\]
for each $i \in \frac{1}{D(V)} \bb Z_{\geq 0}$. The associated graded ring $\overline{R}(V) := \text{gr} \> R(V)$ has $M(V)$ as a basis over $\bb F_q$ and has multiplication rules
\[
x^\mu x^\nu =
\begin{cases}
x^{\mu + \nu} & \text{if $\mu$ and $\nu$ are cofacial with respect to a common closed face of $\Delta(V)$} \\
0 & \text{otherwise}.
\end{cases}
\]

\bigskip\noindent{\bf Toric family.} Let $\bar f(x) = \sum \bar A(\mu) x^\mu \in \bb F_q[x_1^\pm, \ldots, x_n^\pm]$. Define $\Sigma := Supp(\bar f) := \{ \mu \in \bb Z^n \mid \bar A(\mu) \not= 0 \}$. Using the construction above we define $\Delta(\bar f) := \Delta(\Sigma)$, $Cone(\bar f) := Cone(\Sigma)$, $M(\bar f) := M(\Sigma)$, $w := w_\Sigma$, $R := R(\Sigma)$, and $\bar R := \bar R(\Sigma)$. We will assume throughout that
\begin{enumerate}
\item $dim \Delta(\bar f) = n$
\item $\bar f$ is nondegenerate with respect to $\Delta(\bar f)$. (Recall we write for a closed face $\sigma$ of $\Delta(\bar f)$ not containing the origin $\bar f^\sigma = \sum_{\mu \in \sigma} \bar A(\mu) x^\mu$. Then $\bar f$ is \emph{nondegenerate with respect to $\Delta(\bar f)$} if for every closed face $\sigma \in \Delta(\bar f)$ not containing the origin $\{x_1 \frac{\partial \bar f^\sigma}{\partial x_1} , \ldots, x_n \frac{\partial \bar f^\sigma}{\partial x_n}\}$ have no common solutions in $(\overline{\bb F}_q^*)^n$.)
\end{enumerate}

Let $\bar G(x, t) = \bar f(x) + \bar P(x, t) \in \bb F_q[x_1^\pm, \ldots, x_n^\pm, t_1^\pm, \ldots, t_s^\pm]$, with $\bar P(x, t) = \sum t^\gamma \bar P_\gamma(x)$, where $\gamma$ runs over a finite subset $T$ of $\bb Z^s$. We assume, in addition, that
\begin{equation}\label{E: nu}
0 \leq w(\nu) < 1
\end{equation}
for every $\nu \in \bigcup_{\gamma \in T} Supp \bar P_\gamma(x)$. Note, condition (\ref{E: nu}) allows us to assume without loss that all monomials $x^\mu$ in the support of $\bar f$ have weight $w(\mu) = 1$. (All monomials of weight less than 1 can be absorbed by the deforming Laurent polynomial $\bar P$.)

\bigskip\noindent{\bf Relative polytope.} Let $U$ be a finite subset of $\bb Q^s$. As described above, any such set gives rise to data $\Delta(U)$, $Cone(U)$, $M(U)$, and $R(U)$ in $\bb R^s$, and the weight function $w_U$. For simplicity we consider subsets with $Cone(U) = Cone(T)$, $M(U) = M(T)$, with the property
\begin{equation}\label{E: weightbound}
w_U(\gamma) + w(\mu) \leq 1
\end{equation}
for all $(\gamma, \mu) \in \bb Z^s \times \bb Z^n$ in $Supp(\bar P)$. Note that, since we are assuming (\ref{E: nu}), if $U$ is chosen so that $\Delta(U)$ is sufficiently large then (\ref{E: weightbound}) will hold. We now define one which is minimal among all such and therefore gives optimal estimates.

Consider, for any $U$ satisfying (\ref{E: weightbound}), the convex set
\[
\Delta(U) \times \Delta(\bar f) \subset \bb R^{n + s}.
\]
Consider
\begin{equation}\label{E: RelPolyDef}
U_0 := \left\{ \left( \frac{1}{1 - w(\mu)} \right) \gamma \in \bb Q^s \mid  (\gamma, \mu) \in Supp(\bar P) \right\},
\end{equation}
and let $\Gamma := \Delta(U_0)$. Then (\ref{E: weightbound}) implies for any such $\delta = \left( \frac{1}{1 - w(\mu)} \right) \gamma \in U_0$ that $w_U(\delta) \leq 1$. As a consequence $\Gamma \subset \Delta(U)$ and $w_\Gamma(\gamma) \geq w_U(\gamma)$ for all $\gamma \in M(T)$. On the other hand, by very definition, if $\delta = \left( \frac{1}{1 - w(\mu)} \right) \gamma \in U_0$ then $w_\Gamma(\delta) \leq 1$ so that $w_\Gamma(\gamma) + w(\mu) \leq 1$ for all $(\gamma, \mu) \in Supp(\bar P)$, and so (\ref{E: weightbound}) holds for $U = U_0$. Thus the choice of $\Gamma$ above makes $w_\Gamma$ optimal among weight functions $w_U$ satisfying (\ref{E: weightbound}). We call $\Gamma$ the \emph{relative polytope} of the family $\bar G(x, t)$.

We will assume from now on that $U = U_0$ and our weight function $W$ on $M(T) \times M(\bar f)$ is
\begin{equation}\label{E: bigW}
W(\gamma, \mu) := w_\Gamma(\gamma) + w(\mu).
\end{equation}
Note that $W$ is the weight function for the polyhedron $\Gamma \times \Delta(\bar f)$ in $\bb R^{n+s}$. The $\Gamma$ weight of a monomial $(\gamma, \mu)$ in $M(T) \times M(\bar{f})$ is $w_{\Gamma}(\gamma)$. We set $S := \bb F_q[M(\Gamma)]$ which is filtered using $w_\Gamma$. Let $\bar S$ be the associated graded ring. Similarly, set $M := \bb F_q[M(\Gamma) \times M(\bar f)]$, filtered using $W$ defined in (\ref{E: bigW}). $M$ is an $S$-algebra. If $\bar M$ is the associated graded ring $gr(M)$, then $\bar M$ is an $\bar S$-algebra satisfying
\begin{equation}\label{E: Sfilter}
\bar S^{(i)} \bar M^{(j)} \subset \bar M^{(i+j)},
\end{equation}
with multiplication obeying the rule
\begin{equation}\label{E: mult}
(t^\gamma x^\mu)(t^\beta x^\nu) = t^{\gamma + \beta} x^{\mu + \nu}
\end{equation}
if $\gamma$ and $\beta$ are cofacial with respect to a common closed face of $\Gamma$ and $\mu$ and $\nu$ are cofacial with respect to a common closed face of $\Delta(\bar f)$, else the multiplication equals 0.

\bigskip\noindent{\bf Rings of $p$-adic analytic functions.} Let $\zeta_p$ be a primitive $p$-th root of unity. Let $\bb Q_q$ be the unramified extension of $\bb Q_p$ of degree $a := [\bb F_q : \bb F_p]$, and denote by $\bb Z_q$ its ring of integers. Then $\bb Z_q[\zeta_p]$ and $\bb Z_p[\zeta_p]$ are the ring of integers of $\bb Q_q(\zeta_p)$ and $\bb Q_p(\zeta_p)$, respectively. Let $\pi \in \bb Q_p(\zeta_p)$ be a zero of $\sum_{j = 0}^\infty t^{p^j}/p^j$ having $ord_p( \pi ) = 1/(p-1)$. We may have occasion to work over a purely ramified extension of $\bb Q_q$ (containing $\bb Q_q(\zeta_p)$). Say $\bb Q_q(\tilde \pi)$ is a totally ramified extension of $\bb Q_q$ containing $\bb Q_q(\pi)$ with uniformizer $\tilde \pi$, a root of $\pi$. In this situation we denote by $\bb Z_q[\tilde \pi]$ (resp. $\bb Z_p[\tilde \pi]$) the ring of integers of $\bb Q_q(\tilde \pi)$ (resp. $\bb Q_p(\tilde \pi)$), and by $\sigma$ the extension of the Frobenius generator of $Gal(\bb Q_q / \bb Q_p)$ defined by $\sigma(\tilde \pi) = \tilde \pi$. Set
\[
\c O_0 := \left\{ \sum_{\gamma \in M(\Gamma)} C(\gamma) t^\gamma \pi^{w_\Gamma(\gamma)} \mid C(\gamma) \in \bb Z_q[\tilde \pi], C(\gamma) \rightarrow 0 \text{ as } \gamma \rightarrow \infty \right\}.
\]
(We note again that the fractional powers of $\pi$ are to be understood as integral powers of the uniformizer $\tilde \pi$.) Then $\c O_0$ is a ring with a discrete valuation given by
\[
\left| \sum_{\gamma \in M(\Gamma)} C(\gamma) t^\gamma \pi^{w_\Gamma(\gamma)} \right| := \sup_{\gamma \in M(\Gamma)} | C(\gamma) |.
\]
Note, it is not a valuation ring. In fact, the reduction map
\[
\sum_{\gamma \in M(\Gamma)} C(\gamma) \pi^{w_{\Gamma}(\gamma)} t^\gamma  \longmapsto \sum_{\gamma \in M(\Gamma)} \bar C(\gamma) t^\gamma
\]
identifies the rings
\[
\c O_0 / \tilde \pi \c O_0 \xrightarrow{\text{ }\sim\text{ }} \bar S.
\]
Let
\[
\c C_0 := \left\{ \sum_{\mu \in M(\bar f)} \xi(\mu) \pi^{w(\mu)} x^\mu \mid \xi(\mu) \in \c O_0, \xi(\mu) \rightarrow 0 \text{ as } \mu \rightarrow \infty \right\},
\]
an $\c O_0$-algebra. For $\xi \in \c C_0$, write
\[
\xi = \sum_{\mu \in M(\bar f)} \xi(\mu) \pi^{w(\mu)} x^\mu 
= \sum_{(\gamma, \mu) \in M(\Gamma) \times M(\bar f)} \xi(\gamma, \mu) \pi^{W(\gamma, \mu)} t^\gamma x^\mu.
\]
The reduction map mod $\tilde \pi$, taking
\[
\sum_{(\gamma, \mu) \in M(\Gamma) \times M(\bar f)} \xi(\gamma, \mu) \pi^{W(\gamma, \mu)} t^\gamma x^\mu \longmapsto \sum_{(\gamma, \mu) \in M(\Gamma) \times M(\bar f)} \bar \xi(\gamma, \mu) t^\gamma x^\mu
\]
identifies the $\bar S$-algebras
\begin{equation}\label{E: ReductionOfC}
\c C_0 / \tilde \pi \c C_0 \xrightarrow{\text{ }\sim\text{ }} \bar M.
\end{equation}

\bigskip\noindent{\bf The associated complex.} Define $\gamma_0 := 1$, and for $i \geq 1$
\begin{align*}
\gamma_i :&= \pi^{-1} \sum_{j=0}^i \frac{\pi^{p^j}}{p^j} \\
&= - \pi^{-1} \sum_{j=i+1}^\infty \frac{\pi^{p^j}}{p^j}.
\end{align*}
Note that
\[
ord_p(\gamma_i) = \frac{p^{i+1}-1}{p-1} - (i+1).
\]
Let $G(x, t)$ be the lifting of $\bar G(x, t)$ using the Teichm\"uller units for all coefficients. Define
\[
H(x, t) := \sum_{i = 0}^\infty \gamma_i G^{\sigma^i}(x^{p^i}, t^{p^i})
\]
Notice that $G^{\sigma^i}(x^{p^i}, t^{p^i})$ has $W$-weight less than or equal to $p^i$ for all $i \geq 1$, and since
\[
x_l \frac{\partial}{\partial x_l} \left( G^{\sigma^i}(x^{p^i}, t^{p^i}) \right) = p^i \left( x_l \frac{\partial G^{\sigma^i}}{\partial x_l} \right) (x^{p^i}, t^{p^i}),
\]
and
\[
ord_p( \pi \gamma_i p^i ) - \frac{p^i}{p-1} = p^i - 1 \geq 0,
\]
we see that multiplication by $\pi x_l \frac{\partial H(x, t)}{\partial x_l}$ defines an endomorphism of $\c C_0$. Hence, we may define a complex of $\c O_0$-algebras $\Omega^\bullet(\c C_0, \nabla_{ G})$ by
\[
\Omega^i(\c C_0, \nabla_{ G}) := \bigoplus_{1 \leq k_1 < \cdots < k_i \leq n} \c C_0 \frac{dx_{k_1}}{x_{k_1}} \wedge \cdots \wedge \frac{dx_{k_i}}{x_{k_i}}
\]
with boundary map
\[
\nabla_{G}( \eta \frac{dx_{k_1}}{x_{k_1}} \wedge \cdots \wedge \frac{dx_{k_i}}{x_{k_i}} ) = \left( \sum_{l=1}^n D_{l,t}(\eta) \frac{dx_l}{x_l} \right) \wedge \frac{dx_{k_1}}{x_{k_1}} \wedge \cdots \wedge \frac{dx_{k_i}}{x_{k_i}}
\]
with
\[
D_{l,t} = x_l \frac{\partial}{\partial x_l} + \pi x_l \frac{\partial H(x, t)}{\partial x_l}.
\]
In the following section we will see that this complex is acyclic except in top dimension $n$.

\subsection{Computing cohomology of $\Omega^\bullet(\c C_0, \nabla_{\bar G})$}

The complex $\Omega^\bullet(\c C_0, \nabla_{G})$ is a complex of $\c O_0$-algebras. The reduction modulo $\tilde \pi$ of this complex may be identified with the complex $\Omega^\bullet(\bar M, \nabla_{\bar G})$ of $\bar S$-algebras, where
\[
\Omega^i(\bar M, \nabla_{\bar G}) := \bigoplus_{1 \leq k_1 < \cdots < k_i \leq n} \bar M \> \frac{dx_{k_1}}{x_{k_1}} \wedge \cdots \wedge \frac{dx_{k_i}}{x_{k_i}}
\]
and
\[
\nabla_{\bar G}( \bar \eta \frac{dx_{k_1}}{x_{k_1}} \wedge \cdots \wedge \frac{dx_{k_i}}{x_{k_i}} ) = \left( \sum_{l=1}^n \bar D_{l, t}(\eta) \frac{dx_l}{x_l} \right) \wedge \frac{dx_{k_1}}{x_{k_1}} \wedge \cdots \wedge \frac{dx_{k_i}}{x_{k_i}}
\]
where
\[
\bar D_{l,t} := x_l \frac{\partial}{\partial x_l} + x_l \frac{\partial \bar G^{(1)}(x, t)}{\partial x_l}.
\]
Here
\[
\bar G^{(1)}(x, t) := \bar f(x) + \bar P^{(1)}(x, t)
\]
in $\bar M$ consists of the terms in $\bar G$ having $W$-weight precisely equal to 1.

In order to compute the cohomology of the reduced complex it is useful to first compute the cohomology of $\Omega^\bullet(\bar M, \nabla_{d \bar G^{(1)} \wedge})$ where
\[
\Omega^i(\bar M, \nabla_{d \bar G^{(1)} \wedge}) := \bigoplus_{1 \leq k_1 < \cdots < k_i \leq n} \bar M \> \frac{dx_{k_1}}{x_{k_1}} \wedge \cdots \wedge \frac{dx_{k_i}}{x_{k_i}}
\]
with boundary
\[
d \bar G^{(1)} := \sum_{l=1}^n \left( x_l \frac{\partial \bar G^{(1)}}{\partial x_l} \right) \frac{dx_l}{x_l} \wedge.
\]

\begin{theorem}\label{T: 8}
Assume hypotheses 1 and 2 at the beginning of Section \ref{S: Toric Family} above. Then the complex $\Omega^\bullet(\bar M, \nabla_{d \bar G^{(1)} \wedge})$ is acyclic except in top dimension $n$. Furthermore, $H^n(\Omega^\bullet(\bar M, \nabla_{d \bar G^{(1)} \wedge}))$ is a free $\bar S$-algebra of rank equal to $dim_{\bb F_q}( \bar R / \sum_{i=1}^n x_i \frac{\partial \bar f}{\partial x_i} \bar R) = n! vol( \Delta(\bar f))$.

If $\bar B$ is a basis of monomials such that the $\bb F_q$-vector space $\bar V$ spanned by $\bar B$ in $\bar R$ satisfies
\[
\bar R = \bar V \oplus \sum_{i=1}^n x_i \frac{\partial \bar f}{\partial x_i} \bar R,
\]
then
\begin{equation}\label{E: 8a}
\bar M = \bar W \oplus \sum_{i=1}^n x_i \frac{\partial \bar G^{(1)}}{\partial x_i} \bar M
\end{equation}
where $\bar W$ is the free $\bar S$-submodule of $\bar M$ generated by the same set of monomials $\bar B$.
\end{theorem}

\begin{proof}
For the first part of the theorem, we show that for every subset $A$ of $\c S := \{1, 2, \ldots, n\}$ the set $\{ x_i \frac{\partial \bar G^{(1)}}{\partial x_i} \}_{i \in A}$ forms a regular sequence in $\bar M$. If so, then $H^i(\Omega^\bullet(\bar M, \nabla_{d \bar G^{(1)} \wedge})) = 0$ for $0 \leq i < n$ since the complex is in this case the Koszul complex on $\bar M$ defined by the elements $\{ x_i \frac{\partial \bar G^{(1)}}{\partial x_i} \}_{i \in \c S}$. So, assume
\begin{equation}\label{E: G1kernel}
\sum_{i \in A} x_i \frac{\partial \bar G^{(1)}}{\partial x_i} \bar \xi_i(x, t) = 0
\end{equation}
with $\bar \xi_i(x, t)$ in $\bar M$. Since $\bar M$ is graded and the $\{ x_i \frac{\partial \bar G^{(1)}}{\partial x_i} \}_{i \in \c S}$ are homogeneous of $W$-weight 1, it suffices to consider (\ref{E: G1kernel}) in the case in which the $\{ \bar \xi_i(x, t) \}_{i \in A}$ are homogeneous in $\bar M$, say of $W$-weight $k$. We will prove that given (\ref{E: G1kernel}), there exists then a skew-symmetric set $\{\bar \eta_{il}(x, t) \}_{i, l \in A}$ in $\bar M^{(k-1)}$ such that
\begin{equation}\label{E: skewsymm}
\bar \xi_i(x, t) = \sum_{l \in A} x_l \frac{\partial \bar G^{(1)}}{\partial x_l} \bar \eta_{il}(x, t)
\end{equation}
for every $i \in A$. Now we consider for $0 \leq \rho \leq k$ representations of $\{ \bar \xi_i \}_{i \in A}$ of the form
\begin{equation}\label{E: 7b}
\bar \xi_i(x, t) = \sum_{r \geq \rho}^k \bar \xi_i^{(r)}(x, t) + \sum_{l \in A} x_l \frac{\partial \bar G^{(1)}}{\partial x_l} \bar \gamma_{il}(x, t)
\end{equation}
where
\[
\bar \xi_i^{(r)}(x, t) = \sum_{w_\Gamma(\gamma) = r} t^\gamma \bar \xi_{i \gamma}^{(k-r)}(x)
\]
are terms of $\bar M^{(k)}$ with $w_\Gamma$-weight equal to $r$, and $w$-weight equal to $k-r$, and where $\{ \bar \gamma_{i l}(x, t) \}_{i, l \in A}$ is a skew-symmetric subset of $\bar M^{(k-1)}$. Of course, we get such a representation in the case $\rho = 0$ by taking $\bar \gamma_{i l} = 0$ for all $i$ and $l$.

Let $D$ be the least common multiple of $D(\Gamma)$ and $D(\bar f)$ so that
\begin{equation}\label{E: W-weight}
W(\gamma, \mu) = w_\Gamma(\gamma) + w(\mu) \in \frac{1}{D} \bb Z_{\geq 0}.
\end{equation}
We proceed by induction on $\rho$. More precisely, we show that given such a representation (\ref{E: 7b}) for $\{ \bar \xi_i(x, t) \}_{i \in A}$ with a given $\rho = \rho_0$ we may find another such  representation but with $\rho \geq \rho_0 + \frac{1}{D}$, so that in the end we produce a representation with $\rho > k$ hence of the form (\ref{E: skewsymm}).

It is convenient to write
\[
\bar G^{(1)} = \bar f(x) + \sum_{\substack{0 < r \leq 1 \\ r \in (1/D) \bb Z_{\geq 0}}} \bar G_r^{(1)}(x, t)
\]
where
\[
\bar G_r^{(1)}(x, t) = \sum_{w_\Gamma(\gamma) = r} t^\gamma \bar P_\gamma^{(1-r)}(x).
\]
But (\ref{E: G1kernel}) and the skew-symmetry of $\bar \gamma_{il}$ then implies that the sum of terms in $\sum_{i \in A} x_i \frac{\partial \bar G^{(1)}}{\partial x_i} \bar \xi_i(x, t)$ with $w_\Gamma$-weight equal to $\rho_0$ are precisely given by
\[
\sum_{w_\Gamma(\gamma) = \rho_0} t^\gamma \sum_{i \in A} x_i \frac{\partial \bar f}{\partial x_i} \bar \xi_{i \gamma}^{(k-\rho_0)}(x)
\]
so that for each $\gamma \in M(\Gamma)$, $w_\Gamma(\gamma) = \rho_0$, we have
\[
\sum_{i \in A} x_i \frac{\partial \bar f}{\partial x_i} \bar \xi_{i, \gamma}^{(k-\rho_0)}(x) = 0.
\]
Recall \cite{AdolpSperb-ExponentialSumsand-1989} that $\{ x_i \frac{\partial \bar f}{\partial x_i} \}_{i=1}^n$ forms a regular sequence in $\bar R$. Thus, there are for each $\gamma \in M(\Gamma)$ with $w_\Gamma(\gamma) = \rho_0$ a skew symmetric set $\{ \bar \omega_{i, l; \gamma}(x)\}_{i, l \in A} \in \bar R^{(k-\rho_0-1)}$, such that for each $\gamma$, $w_\Gamma(\gamma) = \rho_0$, we have
\begin{equation}\label{E: 12}
\bar \xi_{i, \gamma}^{(k-\rho_0)}(x) = \sum_{l \in A} x_l \frac{\partial \bar f}{\partial x_l} \omega_{i,l; \gamma}(x).
\end{equation}
If we write
\[
\bar \tau_i(x, t) = \bar \xi_i(x, t) - \sum_{l \in A} x_l \frac{\partial \bar G^{(1)}}{\partial x_l} \left( \sum_{w_\Gamma(\gamma) = \rho_0} t^\gamma \bar w_{i,l;\gamma}(x) \right)
\]
then the representation of $\bar \tau_i$ according to $w_\Gamma$-weight begins with at least $\rho_0 + 1/D$. This completes the inductive step and establishes the first part of the theorem.

Since the boundary operator $d \bar G^{(1)} \wedge$ is $\bar S$-linear, $H^n(\Omega^\bullet(\bar M, \nabla_{d \bar G^{(1)} \wedge}))  = \bar M / \sum_{l=1}^n x_l \frac{\partial \bar G^{(1)}}{\partial x_l} \bar M$ is an $\bar S$-module. We claim
\[
\bar M = \bar W \oplus \sum_{l=1}^n x_l \frac{\partial \bar G^{(1)}}{\partial x_l} \bar M
\]
which identifies $H^n(\Omega^\bullet(\bar M, \nabla_{d \bar G^{(1)} \wedge})) $ as the free $\bar S$-submodule $\bar W$ of $\bar M$ with basis $\bar B$. Consider a monomial $t^\gamma x^\mu$ with $W(\gamma, \mu) = k$ in $\bar M$. We know
\[
x^\mu = \sum \bar a_i v_i + \sum_{l=1}^n x_l \frac{\partial \bar f}{\partial x_l} \bar \xi_l(x)
\]
with the $v_i$'s in $\bar B$ of $w$-weight $k - w_\Gamma(\gamma)$ and the $\{ \bar \xi_{l} \}_{l = 1} ^n \subset \bar R^{(k-w_\Gamma(\gamma) - 1)}$. Rewriting we have
\[
x^\mu = \sum \bar a_i v_i + \sum_{l=1}^n x_l \frac{\partial \bar G^{(1)}}{\partial x_l} \bar \xi_l(x) - \sum_{l=1}^n x_l \frac{\partial \bar P^{(1)}(x, t)}{\partial x_l} \bar \xi_l(x).
\]
But then
\[
t^\gamma x^\mu = \sum \bar a_i t^\gamma v_i + \sum_{l=1}^n x_l \frac{\partial \bar G^{(1)}}{\partial x_l} \left( t^\gamma \bar \xi_l(x) \right) - \sum_{l=1}^n x_l \frac{\partial \bar P^{(1)}(x, t)}{\partial x_l}\left( t^\gamma \bar \xi_l(x) \right).
\]
The support of the terms in the last sum on the right all have $W$-weight $k$ but they all have $w_\Gamma$-weight strictly greater than $w_\Gamma(\gamma)$. We may now repeat the preceding argument replacing each $t^{\gamma'} x^{\mu'}$ in the sum $\sum_{l=1}^n x_l \frac{\partial \bar P^{(1)}(x, t)}{\partial x_l} t^\gamma \bar \xi_l(x)$ in an obvious manner and proceed inductively to establish
\begin{equation}\label{E: 13}
\bar M^{(k)} \subset \bar W^{(k)}  + \sum_{l=1}^n x_l \frac{\partial \bar G^{(1)}}{\partial x_l} \bar M^{(k-1)}.
\end{equation}
It remains to show the sum on the right of (\ref{E: 8a}) is direct. Since the submodules on the right side of (\ref{E: 8a}) are homogeneous it suffices to show that the sum on the right side of (\ref{E: 13}) is direct for every $k$. Let
\begin{equation}\label{E: 14}
\bar w = \sum_{l=1}^n x_l \frac{\partial \bar G^{(1)}}{\partial x_l} \bar \xi_l
\end{equation}
in $M^{(k)}$ and write $\bar w$ and $\{ \bar \xi_l \}_{l=1}^n$ in terms of ascending $w_\Gamma$-weight, $\bar w = \sum \bar a_{i,\gamma} t^\gamma v_i$ with $w_\Gamma$-weight of all terms of $\bar w$ at least $\rho_0$, and with at least one coefficient $\bar a_{i, \gamma} \not= 0$ for $(i, \gamma)$ with $w_\Gamma(\gamma) = \rho_0$. Similarly, let the minimal $w_\Gamma$-weight of any term in the support of $\bar \xi_l$ for some $l$ be $\rho_1$. If $\rho_0 < \rho_1$, then (\ref{E: 14}) implies
\[
\sum_i \sum_{w_\Gamma(\gamma) = \rho_0} \bar a_{i, \gamma} t^\gamma v_i = 0
\]
in $\bar M^{(k)}$ which contradicts our assumption that at least one coefficient $\bar a_{i,\gamma} \not= 0$. Next write
\begin{equation}\label{E: 14a}
\bar \xi_l(x, t) = \sum_{w_\Gamma(\gamma) = \rho_1} t^\gamma \bar \xi_{l, \gamma}(x) + \bar \eta_l(x, t)
\end{equation}
with $\bar \eta_l(x, t) \in \bar M^{(k-1)}$ having terms in support all having $w_\Gamma$-weight strictly greater than $\rho_1$. If $\rho_1 < \rho_0$, then (\ref{E: 14}) implies
\[
0 = \sum_{w_\Gamma(\gamma) = \rho_1} t^\gamma \sum_{l=1}^n x_l \frac{\partial \bar f}{\partial x_l} \bar \xi_{l, \gamma}(x).
\]
Thus, for each $\gamma$, $w_\Gamma(\gamma) = \rho_1$,
\[
\sum_{l=1}^n x_l \frac{\partial \bar f}{\partial x_l} \bar \xi_{l, \gamma}(x) = 0.
\]
Since the $\{ x_l \frac{\partial \bar f}{\partial x_l} \}_{l=1}^n$ is a regular sequence in $\bar R$, there is for each $\gamma$, $w_\Gamma(\gamma) = \rho_1$, a skew-symmetric set
\[
\{ \bar \tau_{l,r; \gamma}(x) \}_{l, r \in A} \subset \bar R^{(k-\rho_1-2)}
\]
with
\begin{equation}\label{E: 15}
\bar \xi_{l, \gamma}(x) = \sum_{r=1}^n x_r \frac{\partial \bar f}{\partial x_r} \bar \tau_{l, r; \gamma}(x).
\end{equation}
But then let
\[
\bar \omega_l := \bar \xi_l - \sum_{w_\Gamma(\gamma) = \rho_1} t^\gamma \sum_{r=1}^n x_r \frac{\partial \bar G^{(1)}}{\partial x_r} \bar \tau_{l, r; \gamma}(x).
\]
Using the skew-symmetry and (\ref{E: 14}), we have
\[
\bar w = \sum_{l=1}^n x_l \frac{\partial \bar G^{(1)}}{\partial x_l} \bar \omega_l.
\]
But the smallest $w_\Gamma$-weight among the $\{ \bar \omega_l \}$ is strictly greater than $\rho_1$ by (\ref{E: 15}).  So given (\ref{E: 14}) with $\rho_1 < \rho_0$ we are led to the case where $\rho_1 = \rho_0$. In this case comparing terms having $w_\Gamma$-weight equal to $\rho_0 = \rho_1$ on both sides of (\ref{E: 14}) we have, writing $\bar \xi_l$ as in (\ref{E: 14a}),
\[
\sum_{w_\Gamma(\gamma) = \rho_0, v_i \in \bar B} \bar a_{i,\gamma} v_i = \sum_{w_\Gamma(\gamma) = \rho_0} \sum_{l=1}^n x_l \frac{\partial \bar f}{\partial x_l} \bar \xi_{l, \gamma}(x)
\]
which contradicts the definition of $\bar B$. This completes the proof of Theorem \ref{T: 8}.
\end{proof}

Since $\bar M$ is a graded ring, and the operator $x_l \frac{\partial}{\partial x_l}$ acts on $\bar M$ homogeneously of $W$-weight 0 and multiplication by $x_l \frac{\partial \bar G^{(1)}}{\partial x_l}$ acts homogeneously of $W$-weight 1, then the following result is immediate from Theorem \ref{T: 8}. For notational convenience, we will write $w(v)$ to mean $w(\mu)$ for $v = x^\mu \in \bar B$.

\begin{theorem}\label{T: 16}
We have
\begin{enumerate}
\item $\Omega^\bullet(\bar M, \nabla_{\bar G^{(1)}})$ is acyclic except in top dimension $n$.
\item $H^n(\Omega^\bullet(\bar M, \nabla_{\bar G^{(1)}})) = \bar M / \sum_{l=1}^n \bar D_{l, t} \bar M$ is a free $\bar S$-module of rank $n! vol \Delta(\bar f)$ with basis $\bar B$.
\item We may write
\[
\bar M = \bar W \oplus \sum_{l=1}^n \bar D_{l, t} \bar M.
\]
Furthermore, if $t^\gamma x^\mu \in \bar M^{(k)}$ then
\begin{equation}\label{E: Fq-decomp}
t^\gamma x^\mu = \sum_{v_i \in \bar B} a_{i, \nu} t^\nu v_i + \sum_{l=1}^n \bar D_{l, t} \bar \xi_l(x, t)
\end{equation}
where  $W(\nu, v_i) \leq k$, $w_\Gamma(\nu) \geq w_\Gamma(\gamma)$ and any term $(\beta, \tau)$ in the support of any $\bar \xi_l$ has $W(\beta, \tau) \leq k-1$ and $w_\Gamma(\beta) \geq w_\Gamma(\gamma)$.
\end{enumerate}
\end{theorem}

Finally, a slight modification of Theorem A1 \cite[p.402]{AdolpSperb-ExponentialSumsand-1989}, in which we drop the assumption that $\c O$ is a discrete valuation ring, gives
\begin{theorem}\label{T: A'}
Let $\c O$ be a complete ring under a discrete valuation with uniformizer $\tilde \pi$ and residue ring $F = \c O / \tilde \pi \c O$. Let
\[
C^\bullet = \{ 0 \rightarrow C^0 \xrightarrow{\partial^0} C^1 \xrightarrow{\partial^1} \cdots C^n \rightarrow 0 \}
\]
be a length $n$ cocomplex of flat, separated, complete $\c O$-modules with $\c O$-linear coboundary maps $\partial^i$. Let $\bar C^\bullet$ be the cocomplex obtained by reducing $C^\bullet$ modulo $\tilde \pi$. Then
\begin{enumerate}
\item For any $i$, $H^i(\bar C^\bullet) = 0$ implies $H^i(C^\bullet) = 0$.
\item If $H^n(\bar C^\bullet)$ is a free $F$-module of rank $l$, and $H^{n-1}(\bar C^\bullet) = 0$ then $H^n(C^\bullet)$ is a finite free $\c O$-module of rank $l$.
\end{enumerate}
\end{theorem}

Using Theorem \ref{T: A'}, we obtain the following corollary to Theorem \ref{T: 16}:

\begin{theorem}\label{T: 17}
The complex $\Omega^\bullet(\c C_0, \nabla_{G})$ is acyclic except in top dimension $n$, and $H^n(\Omega^\bullet(\c C_0, \nabla_{G}))$ is a free $\c O_0$-module of rank equal to $n! \> vol \Delta(\bar f)$. Furthermore,
\[
\c C_0 = \sum_{v \in B} \c O_0 \pi^{w(v)} v \oplus \sum_{l=1}^n D_{l,t} \c C_0
\]
where $B$ is a lifting of the monomials in $\bar B$ to characteristic zero and
\[
D_{l, t} := x_l \frac{\partial}{\partial x_l} + \pi x_l \frac{\partial H(x, t)}{\partial x_l}.
\]
\end{theorem}

\subsection{Frobenius}

It will be convenient in this section to denote $\bb F_q$ by $k_0$. To motivate the development of the spaces $\c C_0(\c O_{0,q})$ defined below, formally define
\begin{align}\label{E: FormalFrob}
\alpha_1 &:= \sigma^{-1} \circ \frac{1}{\exp \pi H(x, t^p)} \circ \psi_p \circ \exp \pi H(x, t) \\
\alpha &:= \frac{1}{\exp \pi H(x, t^q)} \circ \psi_q \circ \exp \pi H(x, t), \notag
\end{align}
where
\begin{align*}
\psi_p \left( \sum A(\mu) x^\mu\right) &= \sum A(p \mu) x^\mu \\
\psi_q \left(\sum A(\mu) x^\mu \right) &= \sum A(q\mu) x^\mu,
\end{align*}
and $\sigma \in Gal( \bb Q_q(\zeta_p) / \bb Q_p(\zeta_p) )$ is the Frobenius generator which we extend to $Gal(K / K_0)$ taking $\sigma(\tilde \pi) = \tilde \pi$. Since formally,
\[
D_{l, t} = \frac{1}{\exp \pi H(x, t)} \circ x_l \frac{\partial}{\partial x_l} \circ \exp \pi H(x, t)
\]
the following commutation laws will hold for $l = 1, 2, \ldots, n$,
\begin{equation}\label{E: 30}
q D_{l, t^q} \circ \alpha = \alpha \circ D_{l, t} \quad \text{and} \quad p D_{l, t^p} \circ \alpha = \alpha \circ D_{l, t}.
\end{equation}
Since the differential operators $D_{l, t}$ commute with $\alpha$ by changing $t$ to either $t^p$ or $t^q$, in order to proceed, we need to introduce some new spaces. In the following, $q = p^a$ is an arbitrary power of $p$ (including the case when $a = 0$), so we can handle the cases of $t^q$, $t^p$, and $t$, at the same time. Define
\begin{equation}\label{E: 32}
\c O_{0, q} := \left\{ \sum_{\gamma \in M(\Gamma)} A(\gamma) t^{\gamma} \pi^{w_{q\Gamma}(\gamma)} \mid A(\gamma) \in \bb Z_q[\pi], A(\gamma) \rightarrow 0 \text{ as } \gamma \rightarrow \infty \right\}.
\end{equation}
This ring is the same as $\c O_0$ except using a weight function defined by the dilation $q \Gamma$ (that is, $w_{q \Gamma}(\gamma) = w_\Gamma(\gamma) / q$). We note that here $\c O_{0, 1} = \c O_0$. A discrete valuation may be defined as follows. If $\xi = \sum_{\gamma \in M(\Gamma)} A(\gamma) \pi^{w_{q\Gamma}(\gamma)} t^\gamma$ then
\[
| \xi | := \sup_{\gamma \in M(\Gamma)} |A(\gamma)|.
\]
We may also define the space
\begin{equation}\label{E: 33}
\c C_0(\c O_{0,q}) := \left\{ \sum_{\mu \in M(\bar f)} \xi_\mu x^\mu \pi^{w(\mu)} \mid \xi_\mu \in \c O_{0, q},  \xi_\mu \rightarrow 0 \text{ as } \mu \rightarrow \infty \right\}.
\end{equation}
We will sometimes write $\c C_0$ as $\c C_0(\c O_0)$. For $\eta = \sum_{\mu \in M(\bar f)} \xi_\mu \pi^{w(\mu)} x^\mu$, we set
\[
|\eta| = \sup_{\mu \in M(\bar f)} |\xi_\mu|.
\]
The reduction map
\[
\sum A(\gamma) t^\gamma \pi^{w_{q \Gamma}(\gamma)} \mapsto \sum \overline{A(\gamma)} t^\gamma
\]
takes $\c O_{0,q}$ to the graded ring $\bar S_q := gr \> k_0[M(\Gamma)]$, where $k_0[M(\Gamma)]$ has been graded using $w_{q\Gamma}$. In this case, $\bar S_q$ is identical to $\bar S$ defined earlier but regraded so that $\bar S_q^{(i/q)} = \bar S^{(i)}$.

Replacing the weight function (\ref{E: W-weight}) with
\begin{equation}\label{E: q-weight}
W_q(\gamma, \mu) := w_{q \Gamma}(\gamma) + w(\mu) \in \frac{1}{D(q)} \bb Z_{\geq 0},
\end{equation}
where $D(q)$ is the least common multiple of $D$ and $q$, the proofs of Theorems \ref{T: 8}, \ref{T: 16}, and \ref{T: 17} now may be modified slightly so that analogous versions hold for these spaces, using $M_q := k_0[M(q\Gamma) \times M(\bar f)]$ filtered by $W_q$ and $\bar M_q := gr \> M_q$. Of course, since $M(\Gamma) = M(q \Gamma)$, we see that $\bar {M}_q$ is just $\bar {M}$ regarded with $\bar {M}_q^{(i/q)} = \bar {M}^{(i)}$. We give explicitly the statement of our analogue of Theorem \ref{T: 17}:

\begin{theorem}\label{T: 34}
Let $D_{l, t^q} := x_l \frac{\partial}{\partial x_l} + \pi x_l \frac{\partial H(x, t^q)}{\partial x_l}$. Let $\Omega^\bullet(\c C_0(\c O_{0,q}), \nabla_{\bar G(x, t^q)})$ be the complex with
\[
\Omega^i(\c C_0(\c O_{0, q}), \nabla_{\bar G(x, t^q)}) = \bigoplus_{1 \leq j_1 < \cdots < j_i \leq n} \c C_0(\c O_{0,q}) \frac{dx_{j_1}}{x_{j_1}} \wedge \cdots \wedge \frac{dx_{j_i}}{x_{j_i}}
\]
and
\[
\nabla_{\bar G(x, t^q)}(\xi \frac{dx_{j_1}}{x_{j_1}} \wedge \cdots \wedge \frac{dx_{j_i}}{x_{j_i}}) = \left(\sum_{l=1}^n D_{l, t^q}(\xi) \frac{d x_l}{x_l} \right) \wedge \frac{dx_{j_1}}{x_{j_1}} \wedge \cdots \wedge \frac{dx_{j_i}}{x_{j_i}}.
\]
Then this complex is acyclic except in middle dimension $n$, and $H^n(\Omega^\bullet(\c C_0(\c O_{0, q}), \nabla_{\bar G(x, t^q)}))$ is a free $\c O_{0,q}$-module of rank equal to $n! \> vol(\Delta(\bar f))$. Furthermore,
\[
\c C_0(\c O_{0, q}) = \sum_{v \in B} \c O_{0, q} \pi^{w(v)} v \oplus \sum_{l=1}^n D_{l, t^q} \c C_0(\c O_{0, q}),
\]
where $B$ is the same collection of elements in $M(\bar f)$ as in Theorem \ref{T: 17}.
\end{theorem}

In order to obtain sharp $p$-adic estimates for Frobenius acting on relative cohomology we modify some
of DworkÕs basic constructions. Let $0 < b \leq p/(p-1)$ be a rational number. Define
\begin{align*}
R(b; c) &:= \left\{ \sum_{\gamma \in M(\Gamma)} A(\gamma) t^{\gamma} \mid A(\gamma) \in \bb Q_q(\pi), ord_p \> A(\gamma) \geq b w_\Gamma(\gamma) + c \right\} \\
R(b) &:= \bigcup_{c \in \bb R} R(b; c).
\end{align*}
We define a valuation on $R(b)$ as follows. If $\xi = \sum_{\gamma \in M(\Gamma)} A(\gamma) t^\gamma \in R(b)$, then
\begin{align*}
ord_p \> \xi &= \inf_{\gamma \in M(\Gamma)} \{ ord_p \> A(\gamma) - w_\Gamma(\gamma) b \} \\
&= \sup \{c \in \bb R \mid \xi \in R(b; c) \}.
\end{align*}
Note that
\[
R(b; c) R(b; c') \subset R(b;  c+ c')
\]
since
\[
\left( \sum_{\gamma \in M(\Gamma)} A(\gamma) t^\gamma \right) \left( \sum_{\tilde \gamma \in M(\Gamma)} A'(\tilde \gamma) t^{\tilde \gamma} \right) = \sum_{\beta \in M(\Gamma)} \left( \sum_{\gamma + \tilde \gamma = \beta} A(\gamma) A'(\tilde \gamma) \right) t^\beta.
\]
If $\beta$ is fixed and $(\gamma, \tilde \gamma)$ runs through pairs in $M(\Gamma)$ such that $\gamma + \tilde \gamma = \beta$, then $\sup (|\gamma|, |\tilde \gamma|) \rightarrow \infty$ so that $\inf( |A(\gamma)|, |A'(\tilde \gamma)|) \rightarrow 0$ and the (possibly) infinite series $\sum_{\gamma + \tilde \gamma = \beta} A(\gamma) A'(\tilde \gamma)$ converges. That
\[
\inf \left\{ ord_p \> \left( \sum_{\gamma + \tilde \gamma = \beta} A(\gamma) A'(\tilde \gamma) \right) - w_\Gamma(\beta) b \right\} \geq c + c'
\]
is clear.

For the moment, let $\frak R$ be any ring with a $p$-adic valuation. For $b \in \bb R_{\geq 0}$ and $c \in \bb R$, set
\[
L(b,c; \frak R) := \left\{ \sum_{\mu \in M(f)} \xi_{\mu} x^{\mu} \mid \xi_\mu \in \frak R \text{ and } ord_p(\xi_{\mu}) \geq b w(\mu) + c \right\}.
\]
In particular, if $\frak R = R(b'; c')$ with $c' \geq 0$, we may write $\xi_{\mu} = \sum_{\gamma \in M(\Gamma)} A_{\gamma, \mu} t^{\gamma}$ with $ord_p(A_{\gamma, \mu}) \geq b' w_{\Gamma}(\gamma) + c'$ for all $\gamma$, so that
\[
L(b,c; \frak R) = \left\{ \sum_{(\gamma, \mu) \in M(\Gamma) \times M(f) } A_{\gamma, \mu} t^{\gamma} x^{\mu} \mid ord_p(A_{\gamma, \mu} ) \geq b' w_{\Gamma}(\gamma) + b w(\mu) + c \right\}.
\]
This motivates our definition of the spaces $K(b', b; c)$ below. Let $0 < b', b \leq p/(p-1)$ be rational numbers, and let $c \in \bb R$. Define
\begin{align*}
K(b', b; c) &:= \left\{ \sum_{(\gamma, \mu) \in M(\Gamma) \times M(\bar f)} A_{\gamma, \mu} t^\gamma x^\mu \mid A_{\gamma, \mu} \in \bb Q_q(\pi), ord_p \> A_{\gamma, \mu} \geq b' w_\Gamma(\gamma) + b w(\mu) + c \right\} \\
K(b', b) &:= \bigcup_{c \in \bb R} K(b', b; c).
\end{align*}

We consider now the Frobenius maps $\alpha$ and $\alpha_1$ on $\c C_0$ and on $K(b', b)$. Recall the Artin-Hasse series $E(t) = \exp \left( \sum_{j=0}^\infty t^{p^j}/p^j \right)$ together with $\pi$, a zero of $\sum_{j=0}^\infty t^{p^j}/p^j$ satisfying $ord_p \> \pi = 1/(p-1)$. Dwork's (infinite) splitting function is defined by
\[
\theta(t) := E(\pi t) = \sum_{j = 0}^\infty \theta_j t^j.
\]
Its coefficients satisfy $ord_p \> \theta_j \geq j / (p-1)$. Writing
\[
\bar G(x, t) = \sum \bar A(\gamma, \mu) t^\gamma x^\mu \in \bb F_q[x_1^\pm, \ldots, x_n^\pm, t_1^\pm, \ldots, t_s^\pm],
\]
we let
\[
G(x, t) := \sum A(\gamma, \mu) t^\gamma x^\mu \in \bb Z_q[x_1^\pm, \ldots, x_n^\pm, t_1^\pm, \ldots, t_s^\pm]
\]
be the lifting of $\bar G$ by Teichm\"uller units. Set
\begin{equation}\label{E: 38a}
F(x, t) := \prod_{(\gamma, \mu) \in Supp(\bar G)} \theta( A(\gamma, \mu) t^\gamma x^\mu) \in K(\frac{1}{p-1}, \frac{1}{p-1}; 0) \subset K(b'/p, b/p; 0)
\end{equation}
and
\begin{equation}\label{E: 38b}
F_a(x, t) := \prod_{i=0}^{a-1} F^{\sigma^i}(x^{p^i}, t^{p^i}) \in K(\frac{p}{q(p-1)}, \frac{p}{q(p-1)}; 0) \subset K(b'/q, b/q; 0),
\end{equation}
where $\sigma$ is the Frobenius generator of $Gal(\bb Q_q(\pi) / \bb Q_p(\pi))$ acting on the coefficients of $F$. Note that
\[
\c C_0 \subset K(\frac{1}{p-1}, \frac{1}{p-1}; 0) \subset K(b'/p, b/p)
\]
so that multiplication by $F$ takes $\c C_0$, as well as $K(b'/p, b/p)$, into $K(b'/p, b/p)$, and multiplication by $F_a$ takes these two spaces into $K(b'/q, b/q)$. It is easy to see
\begin{align*}
\psi_p( K(b', b; c)) \subset K(b', pb; c) \\
\psi_q( K(b', b; c)) \subset K(b', qb; c).
\end{align*}
Finally, we note that for $b'$ and $b > 1/(p-1)$, $K(b'/q, b; 0) \subset \c C_0(\c O_{0, q})$. Therefore, since $\psi_p$ acts on the $x$-variables,
\[
\alpha_1 :=  \sigma^{-1} \circ \psi_p \circ F(x, t)
\]
maps $\sigma^{-1}$-semilinearly $\c C_0(\c O_0)$ into $\c C_0( \c O_{0, p})$, and it maps $\sigma^{-1}$-semilinearly $K(b', b; c)$ into $K(b'/p, b; c)$. Similarly, if we define
\[
\alpha := \psi_q \circ F_a(x, t),
\]
then $\alpha$ maps $\c C_0(\c O_0)$ into $\c C_0(\c O_{0, q})$ linearly over $\bb Z_q[\pi]$, as well as $K(b', b; c)$ into $K(b'/q, b; c)$.

We may use $\alpha_1$ and $\alpha$ to define chain maps as follows. Let
\begin{align}
Frob^i &:= \bigoplus_{1 \leq j_1 < \cdots < j_i \leq n} q^{n-i} \alpha \frac{d x_{j_1}}{x_{j_1}} \wedge \cdots \wedge \frac{d x_{j_i}}{x_{j_i}} \label{E: 42a} \\
Frob_1^i &:= \bigoplus_{1 \leq j_1 < \cdots < j_i \leq n} p^{n-i} \alpha_1 \frac{d x_{j_1}}{x_{j_1}} \wedge \cdots \wedge \frac{d x_{j_i}}{x_{j_i}}. \label{E: 42b}
\end{align}
Then the commutation rules (\ref{E: 30}) ensure that these are chain maps:
\begin{align}
\Omega^\bullet(\c C_0(\c O_0), \nabla_{\bar G(x, t)}) &\xrightarrow{\text{ }Frob_1^\bullet\text{ }} \Omega^\bullet(\c C_0(\c O_{0, p}), \nabla_{\bar G(x, t^p)}) \label{E: 43a} \\
\Omega^\bullet(\c C_0(\c O_0), \nabla_{\bar G(x, t)})  &\xrightarrow{\text{ }Frob^\bullet\text{ }} \Omega^\bullet(\c C_0(\c O_{0, q}), \nabla_{\bar G(x, t^q)}) \label{E: 43b}.
\end{align}
All the complexes above are acyclic except in middle dimension $n$.

Define
\begin{align}\label{E: 44}
G_i(x, t^q) &:= x_i \frac{\partial G(x, t^q)}{\partial x_i} \\
G_i^{(1)}(x, t^q) &:= x_i \frac{\partial G^{(1)}(x, t^q)}{\partial x_i},
\end{align}
where $G^{(1)}$ is the Teichm\"uller lifting of $\bar G^{(1)}$. Note that for all $0 < b \leq p/(p-1)$, $\pi G_i^{(1)}(x, t^q)$ and $\pi G_i(x, t^q)$ both belong to $K(b/q, b; -e)$ where $e := b - \frac{1}{p-1}$. In fact, if we write $G = G^{(1)} + g$ then $G(x, t^q) = G^{(1)}(x, t^q) + g(x, t^q)$ and $\pi g_i(x, t^q) \in K(b/q, b; -e + \frac{1}{D})$, where $D$ is the least common multiple of $D(\Gamma)$ and $D(\bar f)$.

\begin{theorem}\label{T: 46}
Let $0 < b \leq p/(p-1)$ and $c \in \bb R$. Then the following equalities hold for any $q$ a power of $p$.
\begin{enumerate}
\item $K(b/q, b; c) = W(b/q, b; c) + \sum_{i=1}^n \pi G_i^{(1)}(x, t^q) K(b/q, b; c + e)$
\item $K(b/q, b; c) = W(b/q, b; c) + \sum_{i=1}^n \pi G_i(x, t^q) K(b/q, b; c + e)$
\end{enumerate}
where
\[
W(b/q, b; c) := \left\{ \sum_{\gamma \in M(\Gamma), v \in B} A(\gamma, v) t^\gamma v \mid A(\gamma, v) \in \bb Q_q(\tilde \pi), ord_p \> A(\gamma, v) \geq b W_q(\gamma, v) + c \right\}
\]
In particular,
\[
W(b/q, b; 0) = K(b/q, b; 0) \cap \bigoplus_{v \in B} \c O_{0, q} \pi^{w(v)} v.
\]
\end{theorem}

\begin{proof}
The right side of these equalities is clearly contained in the left side of the corresponding equality. We first concentrate on the first equality. Let $\xi = \sum_{(\gamma, \mu) \in M(\Gamma) \times M(\bar f)} A(\gamma, \mu) t^\gamma x^\mu \in K(b/q, b; c)$. Let $(\gamma, \mu) \in M(\Gamma) \times M(\bar f)$ with $W_q(\gamma, \mu) = i \in \frac{1}{D(q)} \bb Z_{\geq 0}$, as defined in (\ref{E: q-weight}). We know from (\ref{E: 8a}) that we may write in characteristic $p$
\[
t^\gamma x^\mu = \sum_{v \in \bar B, W_q(\beta, v) = i, w_{q\Gamma}(\beta) \geq w_{q\Gamma}(\gamma)} \overline{C((\gamma, \mu), (\beta, v))} t^\beta v + \sum_{j=1}^n \bar G_j^{(1)}(x, t^q) \bar \eta_j
\]
where $\overline{C((\gamma, \mu), (\beta, v))} \in \bb F_q$ and $\bar \eta_j \in \bar M_q^{(i-1)}$.

We lift the coefficients $\overline{C((\gamma, \mu), (\beta, v))}$ and those of the polynomial $\bar \eta_j$ to characteristic zero 
using (\ref{E: ReductionOfC}) in the form $\c C_0(\c O_{0, q}) / \tilde \pi \c C_0(\c O_{0, q}) \cong \bar M_q$. Then, for each $(\gamma, \mu)$, and $e_0 := \frac{1}{(p-1) \cdot D(q)}$, we have
\[
ord_p\left[ \pi^{W_q(\gamma, \mu)} t^\gamma x^\mu - \left( \sum_{v \in B} C((\gamma, \mu), (\beta, v)) \pi^{W_q(\beta, v)} t^\beta v \right) - \sum_{j=1}^n G_j^{(1)}(x, t^q) \eta_j(\gamma, \mu) \right] \geq e_0.
\]
Since $ord_p\left( \frac{A(\gamma, \mu)}{\pi^{W_q(\gamma, \mu)}} \right) \geq 0$, multiplication by this gives a similar result:
\[
ord_p\left[ A(\gamma, \mu) t^\gamma x^\mu - \left( \sum_{v \in B}  \frac{A(\gamma, \mu)}{\pi^{W_q(\gamma, \mu)}} C((\gamma, \mu), (\beta, v)) \pi^{W_q(\beta, v)} t^\beta v \right) - \sum_{j=1}^n G_j^{(1)}(x, t^q)  \frac{A(\gamma, \mu)}{\pi^{W_q(\gamma, \mu)}} \eta_j(\gamma, \mu) \right] \geq  e_0.
\]
Now,
\[
\omega(\gamma, \mu) := \sum_{v \in B}  \frac{A(\gamma, \mu)}{\pi^{W_q(\gamma, \mu)}} C((\gamma, \mu), (\beta, v)) \pi^{W_q(\beta, v)} t^\beta v \in W(\frac{b}{q}, b; c),
\]
and if we set $\zeta_j := \pi^{-1} \frac{A(\gamma, \mu)}{\pi^{W_q(\gamma, \mu)}} \eta_j(\gamma, \mu)$ then $\zeta_j \in K(\frac{b}{q}, b; c + e)$. Hence,
\[
A(\gamma, \mu) t^\gamma x^\mu - \left( \omega(\gamma, \mu) + \sum_{j=1}^n \pi G_j^{(1)}(x, t^q) \zeta_j(\gamma, \mu) \right) \in K(\frac{b}{q}, b; c + e_0).
\]

We obtain then $\xi = \omega^{(0)} + \sum_{j=1}^n \pi G_j^{(1)}(x, t^q) \zeta_j^{(0)} + \xi^{(1)}$ where $\omega^{(0)} = \sum_{(\gamma, \mu) \in M(\Gamma) \times M(\bar f)} \omega(\gamma, \mu)$, $\zeta_j^{(0)} = \sum_{(\gamma, \mu) \in M(\Gamma) \times M(\bar f)} \zeta_j(\gamma, \mu)$, and $\xi^{(1)} \in K(b/q, b; c+e_0)$. Iterating the argument above, we obtain for every $N \in \bb Z_{\geq 0}$, $\omega^{(N)} \in W(b/q, b; c + N e_0)$, $\zeta_j^{(N)} \in K(b/q, b; c+e+N e_0)$ for $1 \leq j \leq n$ and $\xi^{(N)} \in K(b/q, b; c + N e_0)$ with
\[
\xi^{(N)} = \omega^{(N)} + \sum_{j=1}^n \pi G_j^{(1)}(x, t^q) \zeta_j^{(N)} + \xi^{(N+1)}
\]
so
\[
\xi = \xi^{(N+1)} + \sum_{l=0}^N \omega^{(l)} + \sum_{j=1}^n \pi G_j^{(1)}(x, t^q) \left( \sum_{l=0}^N \zeta_j^{(l)} \right).
\]
Letting $N \rightarrow \infty$, we obtain that the left hand side of the first equality in Theorem \ref{T: 46} is contained in the right hand side.

We now prove the second part of the theorem. Using the first equality of the theorem, we may write each $\xi \in K(b/q, b; c)$ as
\[
\xi = \omega + \sum_{j=1}^n \pi G_j^{(1)}(x, t^q) \zeta _j
\]
with $\omega \in W(b/q, b; c)$ and $\zeta_j \in K(b/q, b; c+e)$. But then
\[
\xi = \omega + \sum_{j=1}^n \pi G_j(x, t^q) \zeta_j - \sum_{j=1}^n \pi g_j(x, t^q) \zeta_j
\]
and $\sum_{j=1}^n \pi g_j(x, t^q) \zeta_j \in K(b/q, b; c + \frac{1}{D(q)})$. Iterating this finishes the proof.
\end{proof}

Recall, we have defined
\[
H(x, t) = \sum_{j=0}^\infty \gamma_j G^{\sigma^j}(x^{p^j}, t^{p^j})
\]
with $ord_p \> \gamma_j = \frac{p^{j+1}}{p-1} - (j+1)$ for $j \geq 0$. Recall also that
\[
D_{l, t^q} := x_l \frac{\partial}{\partial x_l} + \pi H_l(x, t^q),
\]
where $H_l(x, t^q) :=  x_l \frac{\partial}{\partial x_l} H(x, t^q)$.

\begin{theorem}\label{T: 47}
For $b$ a rational number satisfying $1/(p-1) < b \leq p/(p-1)$, $c \in \bb R$, and $q$ a power of $p$, we have
\begin{equation}\label{E: 48}
K(b/q, b; c) = W(b/q, b; c) + \sum_{l=1}^n \pi H_l(x, t^q) K(b/q, b; c + e).
\end{equation}
\end{theorem}

\begin{proof}
It follows from the bound $b \leq p/(p-1)$, together with the $p$-adic order of $\gamma_j$ above, that $\pi H_l(x, t^q) \in K(b/q, b; c + e)$. This gives us that the right hand side of (\ref{E: 48}) is contained in the left hand side. To establish the reverse-inclusion, note that
\[
G_l^{\sigma^j}(x^{p^j}, t^{q p^j}) = G_l(x, t^q)^{p^j} + p h_{l, j}(x, t^q)
\]
where $h_{l, j}(x, t)$ has integral coefficients and all monomials $(\beta, \nu)$ in $h_{l, j}(x, t^q)$ have $W_q$-weight less than or equal to $p^j$. Thus, we may write
\[
\pi H_l(x, t^q) = \pi G_l(x, t^q) Q_l(x, t^q) + K_l(x, t^q)
\]
where
\begin{align*}
Q_l(x, t^q) &:= \sum_{m=0}^\infty \gamma_m p^m G_l(x, t^q)^{p^m-1} \\
K_l(x, t^q) &:= \sum_{m=1}^\infty \gamma_m p^{m+1} h_{l,m}(x, t^q).
\end{align*}
Consequently,
\begin{equation}\label{E: 49}
Q_l(x, t^q), Q_l(x, t^q)^{-1}, K_l(x, t^q) \in K(\frac{p}{q(p-1)}, \frac{p}{p-1}; 0).
\end{equation}
So, if $1/(p-1) < b \leq p/(p-1)$ and $\xi \in K(b/q, b; c)$ then there exists $\omega \in W(b/q, b; c)$ and $\zeta_l \in K(b/q, b; c+e)$ by Theorem \ref{T: 46} such that
\begin{align*}
\xi &= \omega + \sum_{l=1}^n \pi G_l(x, t^q) \zeta_l \\
&= \omega + \sum_{l=1}^n \pi G_l(x, t^q) Q_l(x, t^q) Q_l(x, t^q)^{-1} \zeta_l \\
&= \omega + \sum_{l=1}^n \pi H_l(x, t^q) Q_l(x, t^q)^{-1} \zeta_l - \sum_{l=1}^n K_l(x, t^q) Q_l(x, t^q)^{-1} \zeta_l
\end{align*}
with $Q_l(x, t^q)^{-1} \zeta_l$  and $\sum_{l=1}^n K_l(x, t^q) Q_l(x, t^q)^{-1} \zeta_l$ belonging to $K(b/q, b; c+e)$. Since $b > 1/(p-1)$, we have $e > 0$ and we may proceed recursively as in the previous argument.
\end{proof}

\begin{theorem}\label{T: 50}
For $b$ a rational number satisfying $1/(p-1) < b \leq p/(p-1)$, $c \in \bb R$, $q$ a power of $p$, we have
\[
K(b/q, b; c) = W(b/q, b; c) + \sum_{l=1}^n D_{l, t^q} K(b/q, b; c+e).
\]
\end{theorem}

\begin{proof}
Again, the right hand side is contained in the left side. For the reverse inclusion, let $\xi \in K(b/q, b; c)$.  We know
\[
\xi =  \omega + \sum_{l=1}^n \pi H_l(x, t^q) \zeta_l
\]
where $\omega \in W(b/q, b; c)$ and $\zeta_l \in K(b/q, b; c+e)$ by Theorem \ref{T: 47}. But then
\[
\xi = \omega + \sum_{l=1}^n D_{l, t^q} \zeta_l - \sum_{l=1}^n x_l \frac{\partial \zeta_l}{\partial x_l}.
\]
Since $\sum x_l \frac{\partial \zeta_l}{\partial x_l} \in K(b/q, b; c+e)$, the theorem follows by a similar recursive argument.
\end{proof}

\begin{theorem}\label{T: 51}
For $b$ a rational number satisfying $1/(p-1) < b \leq p/(p-1)$, we have
\[
K(b/q, b)  = W(b/q, b) \oplus \sum_{l=1}^n D_{l, t^q} K(b/q, b).
\]
\end{theorem}

\begin{proof}
By the previous theorem, it only remains to show the sum on the right is direct. Suppose on the contrary $\omega = \sum_{l=1}^n D_{l, t^q} \zeta_l$. Without loss of generality, we may assume both $\omega \in K(b/q, b; 0)$ and each $\zeta_l \in K(b/q,b; e) \subset K(b/q, b; 0)$. Since $b > 1/(p-1)$,
\[
K(b/q, b; 0) \subset \c C_0(\c O_{0,q}),
\]
so that $\omega = \sum_{l=1}^n D_{l, t} \zeta_l$ implies $\omega = 0$ and $\zeta_l = 0$ for every $l$ by Theorem \ref{T: 34}.
\end{proof}

We are now able to provide the following estimates for the entries of the Frobenius. First, note that
\[
\alpha_1(x^\mu) \in K(b/p, b; -(b/p)w(\mu)),
\]
and so by Theorem \ref{T: 50}, we may write
\[
\alpha_1(x^\mu) = \sum_{v \in B} A_{\mu, v} v \qquad \text{mod } \sum_{l=1}^n D_{l, t^p} K(b/p, b)
\]
with
\begin{equation}\label{E: FrobEstimate}
A_{\mu, v} \in L(b/p; (b/p) ( p w(v) - w(\mu) )).
\end{equation}

\subsection{$L$-functions of the toric family}\label{SS: LofToric}

In this section, we will apply Theorem \ref{T: MainDworkTheorem} to prove Theorem \ref{T: IntroMain}.

Let $\lambda \in \overline{\bb F}_q^{*s}$, with $deg(\lambda) := [\bb F_q(\lambda) : \bb F_q]$. Using Dwork's splitting function, define an additive character $\Theta : \bb F_q \rightarrow \overline{\bb Q}_p$ by $\Theta := \theta(1)^{Tr_{\bb F_q / \bb F_p}(\cdot)}$, and $\Theta_\lambda := \Theta \circ Tr_{\bb F_q(\lambda) / \bb F_q}$. Define
\[
S_r(\lambda)  := \sum_{x \in \bb F_{q^{r deg(\lambda)}}^{*n}} \Theta_\lambda \circ Tr_{\bb F_{q^{r deg(\lambda)}} / \bb F_q(\lambda)} \bar G( x, \lambda)
\]
and its $L$-function
\[
L(\bar G_\lambda, T) := L(\bar G_\lambda, \Theta, \bb G_m^n / \bb F_q(\lambda), T) := \exp \left( \sum_{r = 1}^\infty S_r(\lambda) \frac{T^r}{r} \right).
\]
Let $\hat \lambda$ be the Teichm\"uller representative of $\lambda$. Let $\c O_{0, \hat \lambda} = \bb Z_q[\tilde \pi, \hat \lambda]$. There is an obvious ring map, which we call the specialization map at $\hat \lambda$, from $\c O_0$ to $\c O_{0, \hat \lambda}$ induced by the map sending $t \mapsto \hat \lambda$. Similarly, let $\c C_{0, \hat \lambda}$ be the $\c O_{0, \hat \lambda}$-module obtained by specializing the space $\c C_0$ at $t = \hat \lambda$. Let $\alpha_{\hat \lambda} := \alpha^{deg(\lambda)} | _{t = \hat \lambda}$, and define $Frob_{\hat \lambda}^i$ as in (\ref{E: 42a}) but with $\alpha$ replaced by $\alpha_{\hat \lambda}$.  Then using Dwork's trace formula, we obtain
\begin{align*}
S_r(\lambda) &= (q^{r deg(\lambda)} - 1)^n Tr(\alpha_{\hat \lambda} \mid \c C_{0, \hat \lambda}) \\
&= \sum_{i=0}^n (-1)^i Tr( H^i(Frob_{\hat \lambda}^i)^r \mid H^i(\c C_{0, \hat \lambda}, \nabla_{\bar G(x, \hat \lambda)})).
\end{align*}
Since cohomology is acyclic by \cite{AdolpSperb-ExponentialSumsand-1989} except in top dimension $n$, we have
\[
S_r(\lambda) = (-1)^n Tr( H^n(Frob_{\hat \lambda}^n)^r \mid H^n(\c C_{0, \hat \lambda}, \nabla_{\bar G(x, \hat \lambda)})).
\]
In other words, writing $\bar \alpha_{\hat \lambda}$ for $H^n(Frob_{\hat \lambda}^n)$, we have
\begin{align*}
L(\bar G_\lambda, T)^{(-1)^{n+1}} &= det(1 - \bar \alpha_{\hat \lambda} T) \\
&= (1 - \pi_1(\lambda) T) \cdots (1 - \pi_N(\lambda) T),
\end{align*}
where $N = n! \> vol(\Delta_\infty(\bar f))$. In \cite{AdolpSperb-ExponentialSumsand-1989}  it is proved for each such $\lambda$ that the Newton polygon of $L(\bar G_\lambda, T)^{(-1)^{n+1}} $ lies over the Newton polygon (using $ord_{\hat \lambda}$) of
\begin{equation}\label{E: NPfibre}
\prod_{\beta \in \c B} ( 1- (q^{deg(\lambda)})^{w(\beta)} T).
\end{equation}

For each $\lambda \in (\overline{\bb F}_q^\times)^s$, set $\c A(\lambda) := \{\pi_i(\lambda)\}_{i=1}^N$ the collection of eigenvalues of $\bar \alpha_{\hat \lambda}$. Let $\c L$ be a linear algebra operation. Let $\c L \c A(\lambda)$ be the set of eigenvalues of $\c L \bar \alpha_{\hat \lambda}$. Define
\[
L(\c L \c A, \bb G_m^s / \bb F_q, T) := \prod_{\lambda \in |\bb G_m^s / \bb F_q|} \ \    \prod_{\tau(\lambda) \in \c L \c A(\lambda)} (1 - \tau(\lambda) T^{deg(\lambda)})^{-1}.
\]
To aid the reader, we will consider a running example throughout this section; if $\c L$ is the operation of the $k$-th symmetric power tensor the $l$-th exterior power, then
\begin{align*}
\c L \c A(\lambda) &= Sym^k \c A(\lambda) \otimes \wedge^l \c A(\lambda) \\
&= \{ \pi_1(\lambda)^{i_1} \cdots \pi_N(\lambda)^{i_N} \pi_{j_1}(\lambda) \cdots \pi_{j_l}(\lambda) \mid i_1 + \cdots + i_N = k, 1 \leq j_1 < \cdots < j_l \leq N \}.
\end{align*}
Note, the cardinality of $\c L \c A(\lambda)$ is independent of $\lambda$; let $\c L N$ denote this number.

Let $\c B = \{ x^{\mu_1}, \ldots, x^{\mu_N} \}$ be a basis for $H^n := H^n(\Omega^\bullet(\c C_0, \nabla_G))$. For $q$ a power of the prime $p$ (perhaps $q = p^0$) define
\[
\c L H_q^n := \c L H^n( \Omega^\bullet(\c C_0(\c O_{0,q}), \nabla_{G(x, t^q)})).
\]
This is a free $\c O_{0,q}$-module with basis $\c L \c B = \{ e^{\b i} \}_{\b i \in I}$, for some index set $I$. Note, $\c L \c B$ is a basis of $\c L H_q^n$ for every prime power $q$. In our example $\c L H^n = Sym^k H^n \otimes \wedge^l H^n$, elements in the basis $\c L \c B$ take the form
\[
e^{\b i} = (x^{\mu_1})^{i_1} \cdots (x^{\mu_N})^{i_N} \otimes (x^{\mu_{j_1}} \wedge \cdots \wedge x^{\mu_{j_l}})
\]
where
\[
i_1 + \cdots + i_N = k \quad \text{and} \quad 1 \leq j_1 < \cdots < j_l \leq N.
\]
We extend the Frobenius map to this space by defining
\[
\c L \bar \alpha := \c L H^n(Frob^n): \c L H^n \rightarrow \c L H^n_q.
\]
Let $B(t)$ be the matrix of $\c L \bar \alpha$ with respect to the basis $\c L \c B$. Let $\f A(t)$ be the matrix of $\bar \alpha$ with respect to the basis $\c B$, then the matrix of $\bar \alpha_{\hat \lambda}$ is $\f A_{\hat \lambda} := \f A(\hat \lambda^{q^{deg(\lambda) - 1}}) \cdots \f A(\hat \lambda^q) \f A(\hat \lambda)$. Similarly the matrix of $\c L \bar \alpha_{\hat \lambda}$ using the basis $\c L \c B$ is $\c L \f A_{\hat \lambda}$. We have
\[
B(\hat \lambda^{q^{deg(\lambda)-1}}) \cdots B(\hat \lambda^q) B(\hat \lambda) = \c L \f A_{\hat \lambda}.
\]
Since the set of eigenvalues of $\c L \f A_{\hat \lambda}$ is $\c L \c A(\lambda)$, we have
\[
det(1 - B(\hat \lambda^{q^{deg(\lambda)-1}}) \cdots B(\hat \lambda^q) B(\hat \lambda) T^{deg(\lambda)}) = \prod_{\tau(\hat \lambda) \in \c L \c A(\lambda)}(1 - \tau(\hat \lambda) T^{deg(\lambda)}).
\]
Consequently,
\begin{align*}
L(\c L \c A, \bb G_m^s / \bb F_q, T) &:= \prod_{\lambda \in |\bb G_m^s / \bb F_q|} \ \ \prod_{\tau(\hat \lambda) \in \c L \c A(\lambda)}(1 - \tau(\hat \lambda) T^{deg(\lambda)} )^{-1} \\
&= \prod_{\lambda \in | \bb G_m^s / \bb F_q|} det(1 - B(\hat \lambda^{q^{deg(\lambda)}}) \cdots B(\hat \lambda^q) B(\hat \lambda) T^{deg(\lambda)})^{-1} \\
&=: L(B, \bb G_m^s, T)
\end{align*}
as in (\ref{E: LBdef}).

\begin{proposition}
$L( \c L \c A, \bb G_m^s / \bb F_q,  T)$ is a rational function over $\bb Q(\zeta_p)$.
\end{proposition}

\begin{proof}
Writing $L(\c L \c A, \bb G_m^s / \bb F_q, T) = \exp \left( \sum_{m=1}^\infty N_m \frac{T^m}{m} \right)$, we have
\[
N_m = \sum_{\lambda \in (\bb F_{q^m}^*)^s} \sum_{\tau(\lambda) \in \c L \c A(\lambda)} deg(\lambda) \tau(\lambda)^{m / deg(\lambda)}.
\]
Since $\bar G_\lambda$ is nondegenerate for each $\lambda$, by \cite{AdolpSperb-ExponentialSumsand-1989} and \cite{Denef-Loeser-Weights_of_exponential_sums} each eigenvalue has Archimedean weight at most $n$. Embedding into $\bb C$, this means $|\pi_i(\lambda)|_{\bb C} \leq q^{n \> deg(\lambda)/2}$. Since each $\tau(\lambda)$ is a product of $|\c L|$ eigenvalues $\pi_i(\lambda)$, with each factor having weight at most $n$, we have $|N_m|_{\bb C} \leq m \c L N \cdot q^{|\c L| n m/ 2}$. Thus, $L(\c L \c A, \bb G_m^s / \bb F_q, T)$ has a positive radius of convergence over $\bb C$.

Next, since $L(\bar G_\lambda, T)$ is a rational function over $\bb Z[\zeta_p]$, the polynomial $\prod_{\tau(\lambda) \in \c L \c A(\lambda)} (1 - \tau(\lambda) T^{deg(\lambda)})$ has coefficients in $\bb Z[\zeta_p]$. It follows that the coefficients of the power series expansion of $L(\c L \c A, \bb G_m^s / \bb F_q, T)$ lie in a fixed number field $\bb Q(\zeta_p)$. Since $L(\c L \c A, \bb G_m^s / \bb F_q, T)$ is both $p$-adic meromorphic, as in (\ref{E: Dworktrace}) above, and converges on a disc of positive radius over $\bb C$, rationality follows from the Borel-Dwork theorem \cite[Section 4]{Dwork-rationality_of_zeta_function}.
\end{proof}

Define a weight on each basis vector in $\c L \c B = \{ e^{\b i} \}_{\b i \in I}$ as follows. Let
\begin{equation}\label{E: weightBasis}
w(\b i) := \sum_{i=1}^N m_i w(\mu_i)
\end{equation}
where $x^{\mu_i}$ appears $m_i$-times in the basis element $e^{\b i}$. For example, continuing our running example, if
\[
e^{\b i} = (x^{\mu_1})^{i_1} \cdots (x^{\mu_N})^{i_N} \otimes (x^{\mu_{j_1}} \wedge \cdots \wedge x^{\mu_{j_l}})
\]
then
\[
w(\b i) = i_1 w(\mu_1) + \cdots + i_N w(\mu_N) + w(\mu_{j_1}) + \cdots + w(\mu_{j_l}).
\]

\begin{proposition}\label{P: Relate A and B}
There exists a matrix $A = (A_{\b i, \b j})_{\b i, \b j \in I}$ with entries in $L(\frac{1}{p-1})$ such that
\begin{equation}\label{E: BFrobDecomp}
B(t) = A^{\sigma^{a-1}}(t^{p^{a-1}}) \cdots A^\sigma(t^p) A(t) \quad \text{and} \quad A_{\b i, \b j} \in L(1/(p-1) ; \frac{1}{p-1}(p w(\b j) - w(\b i))).
\end{equation}
\end{proposition}

\begin{proof}
From (\ref{E: FormalFrob}), writing $\alpha_1(t)$ for $\alpha_1$ and $\alpha(t)$ for $\alpha$, and writing $\psi_{x, p}$ for $\psi_p$ acting only on the $x$ variables, we have
\begin{align}\label{E: FrobDecomp1}
\alpha(t) &= \psi_{x, p}^a \circ F_a(x, t) \notag\\
&= \psi_{x, p}^a \circ F^{\sigma^{a-1}}(x^{p^{a-1}}, t^{p^{a-1}}) \cdots F^{\sigma}(x^p, t^p) F(x, t) \notag\\
&= \left( \sigma^{-1} \circ \psi_{x, p} \circ F(x, t^{p^{a-1}}) \right) \circ \cdots \circ \left( \sigma^{-1} \circ \psi_{x, p} \circ F(x, t^p) \right) \circ \left( \sigma^{-1} \circ \psi_{x, p} \circ F(x, t) \right) \notag\\
&= \alpha_1(t^{p^{a-1}}) \circ \cdots \circ \alpha_1(t^p) \circ \alpha_1(t).
\end{align}
Denote by $\bar \alpha(t)$ and $\bar \alpha_1(t)$ the maps $H^n(Frob^n)$ and $H^n(Frob_1^n)$, respectively. Then (\ref{E: FrobDecomp1}) shows
\[
\bar \alpha(t) = \bar \alpha_1(t^{p^{a-1}}) \circ \cdots \circ \bar \alpha_1(t^p) \circ \bar \alpha_1(t).
\]
Consequently,
\begin{equation}\label{E: LFrobDecomp}
\c L \bar \alpha(t) = \c L \bar \alpha_1(t^{p^{a-1}}) \circ \cdots \circ \c L \bar \alpha_1(t^p) \circ \c L \bar \alpha_1(t).
\end{equation}
Let $A(t)$ be the matrix of $\c L \bar \alpha_1$ with respect to the basis $\c L \c B$. Then the matrix version of (\ref{E: LFrobDecomp}) is (\ref{E: BFrobDecomp}).

We now proceed to the estimates on $A(t)$. We extend the weight function (\ref{E: q-weight}) as follows
\[
W_q(\gamma, \b i) := w_{q \Gamma}(\gamma) + w(\b i) \in \frac{1}{D(q)} \bb Z_{\geq 0}
\]
for $\gamma \in M(\Gamma)$ and $\b i \in I$. Define the spaces, for $q$ a power of $p$ (perhaps with $q = p^0$) and $c \in \bb R$,
\begin{align*}
\c L \c W(b/q, b; c) &:= \left\{ \sum_{\gamma \in M(\Gamma), \b i \in I} A(\gamma, \b i) t^\gamma e^{\b i} \mid A(\gamma, \b i) \in \bb Q_q(\tilde \pi), ord_p(A(\gamma, \b i)) \geq b W_q(\gamma, \b i) + c \right\} \\
\c L W(b/q, b) &:= \bigcup_{c \in \bb R} \c L W(b/q, b ; c).
\end{align*}
Then, for any rational $b$ satisfying $1/(p-1) < b \leq p/(p-1)$, by Theorem \ref{T: 51} and (\ref{E: FrobEstimate}),
\[
\c L \bar \alpha_1: \c L W(b/p, b/p; 0) \rightarrow \c L W(b/p, b; 0)
\]
and
\[
A_{\b i, \b j} \in L(\frac{b}{p}; \frac{b}{p}( p w(\b j) - w(\b i) )).
\]
Setting $b = p/(p-1)$ to get the best possible $p$-adic estimates, we have
\[
A_{\b i, \b j} \in L(\frac{1}{p-1} ; \frac{1}{p-1}(p w(\b j) - w(\b i))).
\]
\end{proof}

\begin{proof}[Proof of Theorem \ref{T: IntroMain}]
If we modify the basis $\c L \c B$ by the following normalization, $\tilde e^{\bf i} := \pi^{w(\b i)} e^{\b i}$ for each $\b i \in I$, then the matrix of $\c L \bar \alpha_1$ with respect to this basis takes the form $\tilde A(t) = (A_{\b i, \b j} \pi^{w(\b i) - w(\b j)} )$, with entries satisfying
\[
\tilde A_{\b i, \b j} := A_{\b i, \b j} \pi^{w(\b i) - w(\b j)} \in L(1/(p-1) ; w(\b j)).
\]
Then if $\tilde B(t)$ is the matrix of $\c L \bar \alpha$ with respect to this basis, we have
\[
\tilde B(t) = \tilde A^{\sigma^{a-1}}(t^{p^{a-1}}) \cdots \tilde A^\sigma(t^p) \tilde A(t).
\]
We employ Proposition \ref{P: Relate A and B} with the basic data $b = \frac{1}{p-1}$, ramification $e = p-1$, $\{ s(\b j) = (p-1) w(\b j) \}_{j \in I}$. Since $L(B, \bb G_m^s / \bb F_q, T) = L(\tilde B, \bb G_m^s / \bb F_q, T)$, we may apply Theorem \ref{T: MainDworkTheorem} to obtain Theorem \ref{T: IntroMain} and Theorem \ref{T: IntroMainAffine}{\it (a)}. Parts {\it (b)} and {\it (c)} of Theorem \ref{T: IntroMainAffine} follow immediately; see \cite[p.557]{AdolpSperb-Newtonpolyhedraand-1987}. What remains is the determination of $k$ (and consequently $\rho$) in the statement of Theorem \ref{T: MainDworkTheorem}{\it (c)}. We are indebted to Nick Katz for the following argument, which will show that $k = s + n |\c L|$.

Let $\f L_\Theta$ be the $\ell$-adic sheaf on $\bb A^1 / \bb F_q$ corresponding to the character $\Theta$ of $\bb F_q$. Viewing $\bar G$ as a map $\bar G: \bb G_m^n \times \bb G_m^s \rightarrow \bb A^1$ defined over $\bb F_q$, let $\f L_{\Theta(\bar G)}$ be the pullback of $\f L_\Theta$ to $\bb G_m^n \times \bb G_m^s$. Let $\pi_2: \bb G_m^n \times \bb G_m^s \rightarrow \bb G_m^s$ be the projection onto the second factor. It follows that $R^i \pi_{2!} \f L_{\Theta(\bar G)} = 0$ for every $i \not = n$ (since by \cite{AdolpSperb-ExponentialSumsand-1989} and \cite{Denef-Loeser-Weights_of_exponential_sums} every stalk is zero). Further, these same references show that $R^n \pi_{2!} \f L_{\Theta(\bar G)}$ has constant rank equal to $n! \> vol \Delta_\infty(\bar f)$. The sheaf $R^n \pi_{2!} \f L_{\Theta(\bar G)}$ is, in Katz's terminology, of perverse origin \cite[Corollary 6]{Katz-semicontinuity}. Since this sheaf has constant rank, it follows from \cite[Proposition 11]{Katz-semicontinuity} that $R^n \pi_{2!} \f L_{\Theta(\bar G)}$ is lisse. It makes sense therefore to apply a linear algebra operation such as $\c L$ to this sheaf. We view $\c L$ as a quotient of some $r$-fold tensor product;  the minimum such $r$ we denote by $|\c L|$. Then $\c L R^n \pi_{2!} \f L_{\Theta(\bar G)}$ is mixed with weights $\leq |\c L| n$. The eigenvalues of Frobenius acting on $H_{c, \text{\'et}}^i(\bb G_m^s / \bb F_q, \c L R^n \pi_{2!} \f L_{\Theta(\bar G)})$ have weights $\leq i + |\c L| n$ for any $i$ in the range $0 \leq i \leq 2s$. All eigenvalues are algebraic integers so that the weight of any eigenvalue $\gamma$ bounds the valuation:
\[
0 \leq ord_q \gamma \leq \text{ weight}(\gamma).
\]
Thus for $i \leq s$, the $p$-divisibility of any eigenvalue of Frobenius acting on $H^i_{c, \text{\'et}}(\bb G_m^s / \bb F_q, \c L R^n \pi_{2!} \f L_{\Theta(\bar G)})$ satisfies
\[
ord_q \gamma \leq s + |\c L| n.
\]
We now show this inequality holds as well for eigenvalues of Frobenius on $H^i_{c, \text{\'et}}(\bb G_m^s / \bb F_q, \c L R^n \pi_{2!} \f L_{\Theta(\bar G)})$ with $i$ in the upper range $s < i \leq 2s$. For these, we invoke the work of Deligne \cite[Corollary 3.3.3]{Deligne-Weil-II} which applies since $\c L R^n \pi_{2!} \f L_{\Theta(\bar G)}$ is an integral sheaf. For $\gamma$ an eigenvalue of Frobenius acting on $H^i_{c, \text{\'et}}(\bb G_m^s / \bb F_q, \c L R^n \pi_{2!} \f L_{\Theta(\bar G)})$ with $s < i \leq 2s$, Deligne's result implies that $\gamma/q^{i-s}$ is an algebraic integer and pure of some weight. Thus, as above,
\[
ord_q( \gamma / q^{i-s} ) \leq \text{ weight}(\gamma / q^{i-s}).,
\]
so that
\begin{align*}
ord_q(\gamma) - (i-s) &\leq \text{ weight}(\gamma) - 2(i-s) \\
&\leq i + |\c L|n - 2(i-s).
\end{align*}
Thus, $ord_q \> \gamma \leq s + |\c L| n$.
\end{proof}

\section{Other families}

\subsection{Affine and Mixed Toric families}\label{SS: Mixed}

We now state two related theorems which follow from the work and results above in a well-known manner (see \cite{AdolpSperb-ExponentialSumsand-1989} and \cite{Libgober-Sperber-zeta_function_of_monodromy}). Let $\bar f(x) \in \bb F_q[x_1^\pm, \ldots, x_r^\pm, x_{r+1}, \ldots, x_n]$ and set $\c S := \{x_1, \ldots, x_n\}$, $\c S_1 := \{x_1, \ldots, x_r\}$, $\c S_2 := \{ x_{r+1}, \ldots, x_n\}$. We will continue to assume the hypotheses {\it 1.} and {\it 2.}, $dim \Delta_\infty(\bar f) = n$ and $\bar f$ nondegenerate with respect to $\Delta_\infty(\bar f)$, on $\bar f$ from the beginning of Section \ref{S: Toric Family}. We say $\bar f$ is \emph{convenient with respect to $\c S_2$} if for all subsets $A \subset \c S_2$, $dim \Delta_\infty(\bar f_A) = n - |A|$, where $\bar f_A$ is the Laurent polynomial in $n-|A|$ variables obtained by setting each variable $x_i$ with $i \in A$ equal to zero. We will also use the notation $\bar M^{(A)}$ (and $\c C_0^{(A)}$ respectively) for the elements in $\bar M$ (and $\c C_0$) which have support in the set of monomials $x^\mu = x_1^{\mu_1} \cdots x_n^{\mu_n}$ in $M(\bar f)$ satisfying $\mu_i \geq 1$ for every $i \in A$. Then $\bar M^{(A)}$ and $\c C_0^{(A)}$ are ideals in $\bar M$ and $\c C_0$ respectively. Define
\begin{equation}\label{E: 21}
\Omega^i( \c C_0, \c S_2, \nabla_{\bar G} ) := \bigoplus_{A = \{1 \leq j_1 < \cdots < j_i \leq n\}} \c C_0^{(A\cap \c S_2)} \frac{dx_{j_1}}{x_{j_1}} \wedge \cdots \wedge \frac{dx_{j_i}}{x_{j_i}}.
\end{equation}
Since for $\eta \in \c C_0^{(A\cap \c S_2)}$, $l \in \c S$, $D_{l, t}(\eta) \in \c C_0^{((A \cup \{l\}) \cap \c S_2)}$, it follows then that $\nabla_{\bar G}$ defined above defines as well the boundary operator for the subcomplex $\Omega^\bullet( \c C_0, \c S_2, \nabla_{\bar G})$ of $\Omega^\bullet( \c C_0, \nabla_{\bar G})$. In an entirely analogous manner we define subcomplexes $\Omega^\bullet( \bar M, S_2, \nabla_{\bar G})$ of $\Omega^\bullet( \bar M, \nabla_{\bar G})$ and $\Omega^\bullet( \bar M, \c S_2, \nabla_{d \bar G^{(1)}\wedge})$ of $\Omega^\bullet( \bar M, \nabla_{d \bar G^{(1)} \wedge})$ respectively.

\begin{theorem}\label{T: 22}
For $i \not= n$, $H^i(\Omega^\bullet) = 0$ for all three complexes $\Omega^\bullet( \c C_0, \c S_2, \nabla_{\bar G})$, $\Omega^\bullet( \bar M, \c S_2, \nabla_{\bar G})$, and $\Omega^\bullet( \bar M, \c S_2, \nabla_{d \bar G^{(1)} \wedge})$ defined above. Furthermore $H^n(\Omega^\bullet( \c  C_0, \c S_2, \nabla_{\bar G}))$ is a free $\c O_0$-module of rank
\[
\upsilon_{\c S_2}(\bar f) = \sum_{A \subset \c S_2} (-1)^{|A|} (n - |A|)! vol_A( \Delta_\infty(\bar f_A)),
\]
where $vol_A( \Delta_\infty(\bar f_A))$ is the volume with respect to Lebesgue measure on $\bb R^n_A := \{ x = (x_1, \ldots, x_n) \in \bb R^n \mid x_i = 0 \text{ if } i \in A \}$.

Furthermore, we set $\bar B^{(\c S_2)} = \bigcup_{i \in \frac{1}{D}\bb Z_{\geq 0}} \bar B^{(\c S_2, i)}$ where $\bar B^{(\c S_2, i)}$ is a subset of
the set of monomials in $\bar R^{(\c S_2)}$ of weight $i$ such that $\bar V^{(\c S_2, i)}$, the $\bb F_q$-space spanned by $\bar B^{(\c S_2, i)}$, satisfies
\[
\bar R^{(\c S_2, i)} = \bar V^{(\c S_2, i)} \oplus \sum_{l=1}^n x_l \frac{\partial \bar f}{\partial x_l} \bar R^{(\c S_2 - \{l\}, i-1)}.
\]
(Here $\c S_2 - \{l\}$ equals $\c S_2$ if $l \not\in \c S_2$.) Then
\begin{equation}\label{E: 23}
\c C_0^{(\c S_2)} = \sum_{v \in \bar B^{(\c S_2)}} \c O_0 \pi^{w_\Gamma(v)} v \oplus \sum_{l=1}^n D_{l,t} \c C_0^{(\c S_2 - \{l\})}.
\end{equation}
\end{theorem}

As we did in the previous section, let $\lambda \in \overline{\bb F}_q^{*s}$, with $deg(\lambda) := [\bb F_q(\lambda) : \bb F_q]$.  Define the exponential sums
\[
S_l(\lambda)  := \sum_{x \in \bb F_{q^{l deg(\lambda)}}^{*r} \times \bb F_{q^{l deg(\lambda)}}^{n-r}} \Theta_\lambda \circ Tr_{\bb F_{q^{l deg(\lambda)}} / \bb F_q} \bar G( x, \lambda)
\]
and the associated $L$-function on $\bb A^{n-r} \times \bb G_m^{r} / \bb F_q(\lambda)$:
\[
L(\bar G_\lambda, \bb A^{n-r} \times \bb G_m^{r} / \bb F_q(\lambda), T) := \exp \left( \sum_{r = 1}^\infty S_r(\lambda) \frac{T^r}{r} \right).
\]
Let $\hat \lambda$ be the Teichm\"uller representative of $\lambda$. Let $\c O_{0, \lambda}$ be the ring $\bb Z_q[\pi, \hat \lambda]$, and let $\c C_{0, \lambda}$ be the $\c O_{0, \lambda}$-module obtained by specializing the space $\c C_0$ at $t = \hat \lambda$. Let $\alpha_{\hat \lambda} := \alpha^{deg(\lambda)} | _{t = \hat \lambda}$, and define $Frob_{\hat \lambda}^i$ as in (\ref{E: 42a}) but with $\alpha$ replaced by $\alpha_{\hat \lambda}$. Then $Frob_{\hat \lambda}^\bullet$ is a chain map on $\Omega^\bullet(\c C_{0, \lambda}, \c S_2, \nabla_{\bar G(x, \lambda)})$.  Write $\bar \alpha_{\hat \lambda}$ for $H^n(Frob_{\hat \lambda}^n)$ acting on $H^n(\Omega^\bullet(\c C_{0, \hat \lambda}, \c S_2, \nabla_{\bar G(x, \hat \lambda)}))$. Then Theorem \ref{T: 22} says
\begin{align*}
L(\bar G_\lambda, \bb A^{n-r} \times \bb G_m^{r} / \bb F_q(\lambda), T)^{(-1)^{n+1}} &= det(1 - \bar \alpha_{\hat \lambda} T) \\
&= (1 - \pi_1(\lambda) T) \cdots (1 - \pi_{\upsilon_{\c S_2}(\bar f)}(\lambda) T).
\end{align*}

For each $\lambda$, set $\c A_{\c S_2}(\lambda) := \{\pi_i(\lambda)\}_{i=1}^{\upsilon_{\c S_2}(\bar f)}$. Let $\c L$ be a linear algebra operation. Define
\[
L(\c L \c A_{\c S_2}, \bb G_m^s / \bb F_q, T) := \prod_{\lambda \in |\bb G_m^s / \bb F_q|} \prod_{\tau(\lambda) \in \c L \c A_{\c S_2}(\lambda)} (1 - \tau(\lambda) T^{deg(\lambda)})^{-1}.
\]
By a similar method to that of the previous section, we have:

\begin{theorem}\label{T: MixedCase}
For each linear algebra operation $\c L$, the $L$-function $L(\c L\c A_{\c S_2}, \bb G_m^s / \bb F_q, T)$ is a rational function:
\[
L(\c L \c A_{\c S_2}, \bb G_m^s / \bb F_q, T)^{(-1)^{s+1}} = \frac{ \prod_{i=1}^R (1 - \alpha_i T) }{ \prod_{j=1}^S (1 - \beta_j T)} \in \bb Q(\zeta_p)(T).
\]
Further, if we let $N' := \upsilon_{\c S_2}(\bar f)$ then
\begin{enumerate}
\item[(a)] the reciprocal zeros and poles $\alpha_i$ and $\beta_j$ are algebraic integers, and for each $j$, $\beta_j = q^{n_j} \alpha_j$ for some positive integers $n_j$.
\item[(b)] If $\tilde s < s$ then $R = S$, else if $\tilde s = s$ then
\[
0 \leq R - S \leq s! \> vol(\Gamma) \c L N'.
\]
\item[(c)] the total degree is bounded above by
\[
R+S \leq \c L N' \cdot \tilde s! \> vol(\Gamma) \cdot 2^{ \tilde s  + (1 + \frac{1}{\tilde s})n |\c L|} (1 + 2^{1 + \frac{1}{\tilde s}})^{s}.
\]
\item[(d)] If $\bar G \in \bb F_q[x_1, \ldots, x_n, t_1, \ldots, t_s]$, then $L(\c L \c A_{\c S_2}, \bb A^s / \bb F_q, T)$ is a rational function, and writing
\[
L(\c L \c A_{\c S_2}, \bb A^s / \bb F_q, T) = \frac{ \prod_{i=1}^{R'} (1 - \alpha_i T)}{ \prod_{j=1}^{S'}(1 - \beta_j T)} \in \bb Q(\zeta_p)(T),
\]
the zeros and poles satisfy $ord_q(\alpha_i)$ and $ord_q(\beta_j)$ $\geq w(\Gamma) + w(\c L B^{(\c S_2)})$, where $w(\c L B^{(\c S_2)})$ is the minimum weight  of the basis $\c L B^{(\c S_2)}$ (see (\ref{E: weightBasis})). Similar bounds to those found in Theorem \ref{T: IntroMainAffine} for degree and total degree may be given.
\end{enumerate}
\end{theorem}

\subsection{Pure families}\label{SS: Pure}

Let $\lambda \in (\bb F_q^*)^s$. In the case $\c S = \c S_2$, Adolphson and Sperber \cite{AdolpSperb-ExponentialSumsand-1989} observed that in the absolute case that the $L$-function of the exponential sum defined on $\bb A^n$ by $\bar G(x, \lambda)$ was identified with the highest weight factor of the $L$-function of the exponential sum defined by the same $\bar G$ on $\bb G_m^n$. This observation was generalized to more general simplicial toric sums \cite{Adolphson-Sperber-Exp-sums-on-the-torus}. We give now a relative version of this highest weight factor result. Once more let $\bar f \in \bb F_q[x_1^\pm, \ldots, x_n^\pm]$ with $dim \Delta_\infty(\bar f) = n$. Let $\sigma_0$ be the unique face of $\Delta_\infty(\bar f)$ containing the origin which spans a linear subspace of smallest dimension; denote this dimension by $dim(\sigma_0)$. If the origin is an interior point of $\Delta_\infty(\bar f)$ then $\sigma_0 = \Delta_\infty(\bar f)$. We say $\Delta_\infty(\bar f)$ is {\it simplicial} with respect to the origin if $\sigma_0$ is contained in exactly $n - dim(\sigma_0)$ faces of codimension $n - 1$ of $\Delta_\infty(\bar f)$. This is always satisfied when $n=2$. It is also always true when the origin is an interior point of $\Delta_\infty(\bar f)$. Note that if $\Delta_\infty(\bar f)$ is simplicial with respect to the origin, then this holds as well for $\bar G$.

Using notation which is consistent with the previous section, we will write $dim(\sigma_0) = n - r$. Let the equations of the hyperplanes $\{ H_i\}_{i=r+1}^{n}$ spanned by each of the codimension one faces $\{ \sigma_i \}_{i=r+1}^{n}$ of $\Delta_\infty(\bar f)$ containing the origin be given by
\[
l_i(x_1, \ldots, x_n) := \sum_{j=1}^n a_{i,j} x_j = 0 \qquad r+1 \leq i \leq n,
\]
where $a_{i,j} \in \bb Z$ and for each $r+1 \leq i \leq n$, $gcd(a_{i,r+1}, \ldots, a_{i,n}) = 1$. We assume the inequalities $l_i(x_1, \ldots, x_n) \geq 0$ define $Cone(\bar f)$. Let $L$ be the greatest common divisor of all the $(n-r) \times (n-r)$ subdeterminants of the $(n-r) \times n$ integer matrix $A := (a_{i,j})_{r+1 \leq i \leq n, 1 \leq j \leq n}$. If $p \nmid L$ then we may find $n - r$ columns, say the last $n - r$ columns for convenience of notation, so that the matrix $A$ may be written in block form as $(A_1 \mid A_2)$ with $A_2$ a square $(n-r) \times (n-r)$ matrix with integer determinant which is relatively prime to $p$.

Write $\c S = \{1, \ldots, n\}$, $\c S_1 = \{1, \ldots, r\}$, and $\c S_2 = \{r+1, \ldots, n\}$. Let $I$ be the $r \times r$ identity matrix and let
\[
\tilde A :=
\left(
\begin{array}{cc}
  I & 0   \\
  A_1 & A_2
\end{array}
\right),
\]
an $n \times n$ matrix with entries $\tilde a_{i,j}$, so $\tilde a_{i, j} = a_{i, j}$ for $i \geq r+1$. For $i \in \c S$, write
\[
l_i(x_1, \ldots, x_n) := \sum_{j=1}^n \tilde a_{i,j} x_j
\]
and set
\begin{equation}\label{E: Ds}
\c D_{i, t} := 
\begin{cases}
D_{i, t} & \text{for } i < r+1 \\
\sum_{j=1}^n \tilde a_{ij} D_{j, t} & \text{for } i \geq r+1
\end{cases}
\qquad \text{and} \qquad 
\bar{\c D}_{i, t} =
\begin{cases}
\bar D_{i, t} & \text{for } i < r+1 \\
\sum_{j=1}^n \bar a_{ij} \bar D_{j, t} & \text{for } i \geq r+1.
\end{cases}
\end{equation}
Define the complex $\Omega^\bullet(\c C_0, \tilde \nabla_{\bar G})$ as follows. It has the same spaces as those of the complex $\Omega^\bullet(\c C_0, \nabla_{\bar G})$ but with boundary map on $\Omega^i$ defined by (\ref{E: Ds}):
\[
\tilde \nabla_{\bar G}( \eta \frac{dx_{k_1}}{x_{k_1}} \wedge \cdots \wedge \frac{dx_{k_i}}{x_{k_i}}) := \left( \sum_{l=1}^n \c D_{l,t}(\eta) \frac{d x_l}{x_l} \right) \wedge \frac{dx_{k_1}}{x_{k_1}} \wedge \cdots \wedge \frac{dx_{k_i}}{x_{k_i}}.
\]
The reduction modulo $\pi$ of this complex is the complex $\Omega^\bullet(\bar M, \tilde \nabla_{\bar G})$ with the same $\Omega^i$ space as $\Omega^\bullet(\bar M, \nabla_{\bar G})$ but with boundary map
\[
\tilde \nabla_{\bar G}( \eta \frac{dx_{k_1}}{x_{k_1}} \wedge \cdots \wedge \frac{dx_{k_i}}{x_{k_i}}) := \left( \sum_{l=1}^n \bar{\c D}_{l,t}(\eta) \frac{d x_l}{x_l} \right) \wedge \frac{dx_{k_1}}{x_{k_1}} \wedge \cdots \wedge \frac{dx_{k_i}}{x_{k_i}}.
\]

\begin{theorem}\label{T: 25}
If $p \nmid L$, then $\Omega^\bullet(\c C_0, \tilde \nabla_{ G})$ and $\Omega^\bullet(\c C_0, \nabla_{G})$ are isomorphic as $\c O_0$-modules, and $\Omega^\bullet(\bar M, \tilde \nabla_{\bar G})$ and $\Omega^\bullet(\bar M, \nabla_{\bar G})$ are isomorphic as $\bar S$-algebras.
\end{theorem}

For $A \subset \c S_2$, we define $\c C_0^{(A)}$ (and $\bar M^{(A)}$ respectively) to be the ideal of elements in $\c C_0$ (respectively in $\bar M$) with support in the monomials $x^\mu = x_1^{\mu_1} \cdots x_n^{\mu_n}$ such that $l_i(\mu) \geq 1$ for all $i \in A$. Then for $A \subset \c S$, $l \in \c S$,
\begin{equation}\label{E: 26}
\c D_{l, t} \c C_0^{(A \cap \c S_2)} \subset \c C_0^{((A \cup \{l\}) \cap \c S_2)} \qquad \text{and} \qquad
\bar{\c D}_{l,t} \bar M^{(A \cap \c S_2)}) \subset \bar M^{((A \cup \{l\}) \cap \c S_2)}.
\end{equation}
We proceed in a manner entirely analogous to Section \ref{SS: Mixed}.

In order to simplify the notation we will denote by $S \tilde \Omega^\bullet(\c C_0)$ the subcomplex of $\Omega^\bullet(\c C_0, \tilde \nabla_{\bar G})$ defined by
\begin{equation}\label{E: 27}
S \tilde \Omega^i := \bigoplus_{A = \{ 1 \leq j_1 < \cdots < j_i \leq n\}} \c C_0^{(A\cap \c S_2)} \frac{d x_{j_1}}{x_{j_1}} \wedge \cdots \wedge \frac{d x_{j_i}}{x_{j_i}}.
\end{equation}
Define in a similar way the subcomplex $S \tilde \Omega^\bullet(\bar M)$ of $\Omega^\bullet(\bar M, \nabla_{\bar G})$. Note, these are subcomplexes because of (\ref{E: 26}).

\begin{theorem}\label{T: 28}
Assume $\bar f$ is nondegenerate with respect to $\Delta_\infty(\bar f)$ and that $\Delta_\infty(\bar f)$ is simplicial with respect to the origin. Assume $p \nmid L$. Then $H^i(S \tilde \Omega^\bullet(\c C_0))$ and $H^i( S \tilde \Omega^\bullet(\bar M))$ are acyclic except in top dimension $n$. Furthermore, $H^n( S \tilde \Omega^\bullet(\c C_0))$ is a free $\c O_0$-module of rank
\[
\upsilon_{\c S_2}(\bar f) = \sum_{A \subset \c S_2} (-1)^{|A|} (n- |A|)! vol_A( \Delta_\infty(\bar f_A))
\]
 where here we let $H_A$ be the intersection of the hyperplanes $\{H_i = 0 \}$ for $i \in A$ and $vol_A(\Delta_\infty(\bar f_A))$ is the volume with respect to Haar measure in $H_A$ normalized so that the fundamental domain for the lattice $\bb Z^n \cap H_A$ in $H_A$ has measure 1.

Let $\bar B^{(\c S_2)} = \bigcup_{i \in \frac{1}{D}\bb Z_{\geq 0}} \bar B^{(\c S_2, i)}$ where $\bar B^{(\c S_2, i)}$ is a subset of the monomials in $\bar R^{(\c S_2, i)} = \bar R^{(\c S_2)} \cap \bar R^{(i)}$ such that $\bar V^{(\c S_2, i)}$ the $\bb F_q$-space spanned by $\bar B^{(\c S_2, i)}$ satisfies
\[
\bar R^{(\c S_2, i)} = \bar V^{(\c S_2, i)} \bigoplus \sum_{j=1}^n l_j(x_1 \frac{\partial \bar f}{\partial x_1}, \ldots, x_n \frac{\partial \bar f}{\partial x_n}) \bar R^{(\c S_2 - \{j\}, i)}
\]
(here $\c S_2 - \{j\} = \c S_2$ if $i \not\in \c S_2$). Then
\[
\c C_0^{(\c S_2)}(\frac{1}{p-1}) = \sum_{v \in \bar B^{(\c S_2)}} \c O_0 \pi^{w_\Gamma(v)} v \oplus \sum_{j=1}^n \c D_{j, t} \c C_0^{(\c S_2 - \{j\})}(\frac{1}{p-1}).
\]
\end{theorem}

We proceed in a manner quite analogous to the case in the previous section. For each $\lambda \in (\overline{\bb F}_q^*)^s$, let
\[
L(\bar G_\lambda, \bb G_m^r, T)^{(-1)^{n+1}} = \prod_{i=1}^N (1 - \pi_i(\lambda) T)
\]
with $N := n! \> vol \Delta_\infty(\bar f)$. Set $\c A(\lambda) = \{ \pi_i(\lambda)\}_{i=1}^N$ and let $\c W_n(\lambda)$ be the subset of $\c A(\lambda)$ consisting of reciprocal zeros of highest archimedean weight:
\[
\c W_n(\lambda) := \{ \pi(\lambda) \in \c A(\lambda) \mid | \pi(\lambda)| = q^{deg(\lambda)n/2} \}.
\]
The set $\c W(\lambda)$ now plays precisely the same role as $\c A_{\c S_2}(\lambda)$ in the previous section. For a linear algebraic operation $\c L$, set
\[
L(\c L \c W_n, \bb G_m^s / \bb F_q, T) := \prod_{\lambda \in |\bb G_m^s / \bb F_q|} \prod_{\tau(\lambda) \in \c L \c W_n(\lambda)} (1 - \tau(\lambda) T^{deg(\lambda)})^{-1}.
\]
Since Galois action preserves weight, this is a rational function over $\bb Q(\zeta_p)$ with properties as follows.

\begin{theorem}
For each linear algebraic operation $\c L$, the $L$-function $L(\c L \c W_n, \bb G_m^s / \bb F_q, T)$ is a rational function over $\bb Q(\zeta_p)$ with estimates for its degree, total degree, and for the $p$-divisibility of its reciprocal zeros and poles are precisely the same as those in Theorem \ref{T: MixedCase}. Here $\upsilon_{\c S_2}(\bar f)$ is the same alternating sum of volumes as in Theorem \ref{T: MixedCase} and here $w(\c L \c B^{(\c S_2)})$ is the minimum of the weights in the basis $\c L \c B^{(\c S_2)}$.
\end{theorem}

\subsection{Unit root $L$-function}\label{SS: Unitfam}

Let
\[
\bar G(x, t) := \bar f(x) + \bar P(x, t) \in \bb F_q[x_1^\pm, \ldots, x_n^\pm, t_1, \ldots, t_s]
\]
where $\bar f(x)$ is nondegenerate with respect to $\Delta_\infty(\bar f)$. Let $\bar G$ satisfy the hypotheses of the toric family in Section \ref{S: Toric Family}, that is, $dim \> \Delta_\infty(\bar f) = n$, $\bar f$ is nondegenerate with respect to $\Delta_\infty(\bar f)$, and $0 \leq w(\mu) < 1$ for every $x^\mu$ in $Supp(\bar P)$.

For each $\lambda  \in \overline{\bb F}_q^{*s}$ the $L$-function $L( \bar G_\lambda, \Theta, \bb G_m^n / \bb F_q(\lambda), T)$ has a unique unit root, say $\pi_0(\lambda)$. Define the $k$-th moment unit root $L$-function by
\[
L_{\text{unit}}( k, \bar G, \bb A^s / \bb F_q, T) := \prod_{\lambda \in |\bb A^s / \bb F_q|} (1 - \pi_0(\lambda)^k T^{deg(\lambda)})^{-1}.
\]
This is a meromorphic function by Wan's theorem \cite[Theorem 8.4]{Wan-Dwork'sconjectureunit-1999} and so may be written as
\[
L_{\text{unit}}( k, \bar G, \bb A^s / \bb F_q, T) = \frac{\prod_{i=1}^\infty (1 - \alpha_i T)}{\prod_{j=1}^\infty (1 - \beta_j T)} \quad \text{with} \quad \alpha_i, \beta_j \rightarrow 0 \text{ as } i, j \rightarrow \infty.
\]
Recall the maps $\bar \alpha_1$, $\bar \alpha$, and the basis $\c B$ from Section \ref{SS: LofToric}. Let $\f A_1(t)$ and $\f A(t)$ be the matrices of $\bar \alpha_1$ and $\bar \alpha$ with respect to $\c B$. Then
\[
\f A(t) = \f A_1^{\sigma^{a-1}}(t^{p^{a-1}}) \cdots \f A_1^\sigma(t^p) \f A_1(t),
\]
where $\f A_1$ has entries in $L(1/(p-1))$. As in \cite{Wan-Dwork'sconjectureunit-1999}, $\f A(t)$ defines a nuclear $\sigma$-module $\phi$ ordinary at slope zero. If $\phi_0$ is the rank one unit root $\sigma$-module coming from the Hodge-Newton decomposition of $\phi$, then
\[
L(\phi_0^k, \bb A^s / \bb F_q, T) = L_{\text{unit}}(k, \bar G, \bb A^s / \bb F_q, T).
\]
By Theorem \ref{T: unitrankone}, we have

\begin{theorem}
Assume $\Gamma \subset \bb R^s_{\geq 0}$. Then for every $i$ and $j$, $ord_q(\alpha_i)$ and $ord_q(\beta_j) \geq w(\Gamma)$.
\end{theorem}

\bibliographystyle{amsplain}
\bibliography{References}

\end{document}